\documentclass[11pt]{article}    
\usepackage{amsmath, amssymb, amsthm,pdfsync}
\usepackage{fullpage}
\usepackage{verbatim}
\usepackage{graphicx}
\usepackage{subfig, caption}
\usepackage{cancel}
\usepackage{enumitem}
\usepackage{wrapfig}
\usepackage{mathtools}
\usepackage{xcolor}
\usepackage{epigraph}
\usepackage{cleveref}
\usepackage{overpic}
\usepackage{esint}
\usepackage{dsfont}
\usepackage{color}

\newcommand{\pdr}[2]{\dfrac{\partial{#1}}{\partial{#2}}}

\newcommand{\farc}{\frac}

\setlist[itemize]{itemsep=-1mm}

\newtheorem{theorem}{Theorem}
\newtheorem{lemma}[theorem]{Lemma}
\newtheorem{lem}[theorem]{Lemma}
\newtheorem{prop}[theorem]{Proposition}

\newtheorem{proposition}[theorem]{Proposition}

\newcommand{\vertiii}[1]{{\left\vert\kern-0.25ex\left\vert\kern-0.25ex\left\vert #1 
    \right\vert\kern-0.25ex\right\vert\kern-0.25ex\right\vert}}

\newcommand{\R}{\mathbb{R}}
\newcommand{\Z}{\mathbb{Z}}
\newcommand{\cC}{\mathcal{C}}

\newcommand{\eps}{\varepsilon}

\newcommand{\1}{\mathds{1}}
\newcommand{\ds}{\displaystyle}

\newcommand{\Rm}{{\mathbb R}}

\numberwithin{equation}{section}
\numberwithin{theorem}{section}

\Crefname{assumption}{Assumption}{Assumptions}
\Crefname{theorem}{Theorem}{Theorems}
\Crefname{lem}{Lemma}{Lemmas}
\Crefname{cor}{Corollary}{Corollaries}
\Crefname{prop}{Proposition}{Propositions}
\Crefname{theorem}{Theorem}{Theorems}
\Crefname{conjecture}{Conjecture}{Conjectures}

\begin{document}
\title{The Bramson logarithmic delay in the cane toads equations}

%

\author{Emeric Bouin 
\footnote{CEREMADE - Universit\'e Paris-Dauphine, UMR CNRS 7534, Place du Mar\'echal de Lattre de Tassigny, 75775 Paris Cedex 16, France. E-mail: \texttt{bouin@ceremade.dauphine.fr}}\and  
Christopher Henderson \footnote{Ecole Normale Sup\'erieure de Lyon, UMR CNRS 5669 'UMPA', 46 all\'ee d'Italie, F-69364~Lyon~cedex~07, France. E-mail: \texttt{christopher.henderson@ens-lyon.fr}}\and
Lenya Ryzhik \footnote{Department of Mathematics, Stanford University, Stanford, CA 94305, E-mail: \texttt{ryzhik@math.stanford.edu}}}

\maketitle
\begin{abstract}
We study a nonlocal reaction-diffusion-mutation equation modeling 
the spreading of a cane toads
population structured by a phenotypical trait responsible for 
the spatial diffusion rate. When the trait space is bounded,
the cane toads equation admits traveling wave solutions~\cite{BouinCalvez}.  
Here, we prove a Bramson type spreading result: the lag between the 
position of solutions with localized initial data and that of the 
traveling waves grows as $(3/(2\lambda^*))\log t$. This result relies on a
present-time Harnack inequality which allows to compare solutions of the cane
toads equation to those of a Fisher-KPP type equation that is local in the
trait variable.
\end{abstract}
%

\section{Introduction}

\subsection*{The cane toads spreading}

Cane toads were introduced in Queensland, Australia in 1935,
to control the native cane beetles in sugar-cane fields.
Initially, about one hundred cane toads were released, and by now,
their population is estimated to be about two hundred million,   
leading to disastrous ecological effects.
Their invasion has interesting 
features different from the standard spreading observed in 
most other species~\cite{phillips2006invasion}. 
Rather than invade at a constant speed,
the annual rate of progress of the toad invasion front has 
increased by a factor of about five since the toads were first introduced:
the toads expanded their range by about
10~km a year during the 1940s to the 1960s, but were invading new
areas at a rate of over 50 km a year by 2006. Toads with 
longer legs move faster and are the first to arrive to new areas,
followed later by those with shorter legs. In addition, 
those at the front have longer legs than toads in the long-established
populations --  the typical leg length of the advancing population
at the front grows in time.
The leg length is greatest in the new
arrivals and then declines over a sixty year period.  
The cane toads are just one example of  a non-uniform space-trait 
distribution -- one other is  the expansion of the bush crickets in 
Britain~\cite{Thomas}. There, the difference is 
between the long-winged and short-winged crickets,
with similar conclusions.   In all such phenomena, modelling of the spreading rates has
to include the trait structure of the population. 

\subsection*{The cane toads equation}
 
We consider here a model of the cane toads invasion proposed in~
\cite{BenichouEtAl},
based on the classical Fisher-KPP equation~\cite{Fisher,KPP}.  The population density $n(t,x,\theta)$ 
is structured by a spatial variable~$x$, and a 
motility variable $\theta$. 
This population undergoes diffusion in the trait variable $\theta$, with a constant diffusion coefficient, 
representing mutation, and in the spatial variable, with the diffusion coefficient~$\theta$, representing the effect of the trait  on the spreading 
rates of the species.  In addition, each toad competes locally in space with all other individuals for resources. 
If the competition is local in the trait variable, 
then the corresponding Fisher-KPP model is
\begin{equation}\label{nov1004}
u_t=\theta u_{xx}+u_{\theta\theta}+u(1-u).
\end{equation}
It is much more biologically relevant to consider a non-local in trait competition (but still local in space), which leads to 
\begin{equation}\label{nov1006}
n_t=\theta n_{xx}+ n_{\theta\theta}+rn(1-\rho),
\end{equation}
where
\begin{equation}\label{nov1008}
\rho(t,x)=\int_{\Theta}  n(t,x,\theta)d\theta
\end{equation}
is the total population at the position $x$. Here, $\Theta$ is the set of all possible traits. It is either an infinite 
semi-interval:~$\Theta=[\underline\theta,+\infty)$, or an interval $\Theta=[\underline\theta,\overline\theta]$.   For simplicity,
we consider the one-dimensional case: $x\in\Rm$. 
Both (\ref{nov1004}) and (\ref{nov1006}) are supplemented by  Neumann boundary conditions at $\theta=\underline\theta$
and $\theta=\overline\theta$ (in the case when $\Theta$ is a finite interval):
\begin{equation}\label{nov1012}
n_\theta(t,x,\underline\theta)=n_\theta(t,x,\overline\theta)=0,~~t>0,~x\in\Rm.
\end{equation}

The cane toads equation is but one example among other non-local reaction 
models  that have been extensively studied recently~\cite{AlfaroCovilleRaoul,BerJinSil,BouinMirrahimi,FayeHolzer, HamelRyzhik,NadinPerthameTang,NadinRossiRyzhikPerthame}.  Mathematically, non-local models are 
particularly interesting since their solutions do not obey the maximum principle and standard propagation
results for the scalar local reaction-diffusion equations do not apply. Rather, on the qualitative level they behave as
solutions of systems of reaction-diffusion equations, for which much fewer spreading results are available.
The study of the spreading of solutions to the
cane toads equations started with 
a Hamilton-Jacobi framework that was formally developed in~\cite{BCMetal}, and rigorously justified  in~\cite{Turanova}
when $\Theta$ is a finite interval. 
Existence of the travelling waves for (\ref{nov1006})  in that case has been proved in
\cite{BouinCalvez}.

 As far as unbounded traits are concerned, a formal argument 
in~\cite{BCMetal} using a Hamilton-Jacobi framework predicted front acceleration, observed in the field,
and the spreading rate of $O(t^{3/2})$.
A rigorous proof of this spreading rate has been given in~\cite{BerestyckiMouhotRaoul,BoHR}. 
%

\subsection*{The main results}
 
In this paper, we consider the spreading rate of the solutions of the non--local cane toads equation~(\ref{nov1006})-(\ref{nov1008}),
with $x\in\Rm$ and $\theta\in\Theta=[\underline\theta,\overline\theta]$, and the Neumann boundary conditions  (\ref{nov1012}).
The initial condition $n(0,x,\theta)=n_0(x,\theta)\not\equiv 0$ is non-negative and has localized support in a sense to be made precise later. The classical result
of~\cite{Fisher,KPP} says that solutions of the scalar KPP equation 
\begin{equation}\label{aug2902}
v_t=v_{xx}+v(1-v)
\end{equation}
with a non-negative compactly supported initial condition  $v_0(x)=v(0,x)$
propagate with the speed~$c^*=2$ in the sense that
\begin{equation}\label{aug2904}
\lim_{t\to+\infty} v(t,ct)=0,
\end{equation}
for all $c>c^*$, and
\begin{equation}\label{aug2905}
\lim_{t\to+\infty} v(t,ct)=1,
\end{equation}
for all $c\in[0,c^*]$. The corresponding result for the solutions of (\ref{nov1006}) follows from the Hamilton-Jacobi limit in~\cite{Turanova}.
The Fisher-KPP result  for the solutions of (\ref{aug2902})
has been refined by Bramson in~\cite{Bramson78,Bramson83}. He has shown the following: for any $m\in(0,1)$, let
\[
X_m(t)=\sup\{x:~v(t,x)=m\},
\]
with $s\in (0,1)$. This level set has the asymptotics
\begin{equation}\label{aug2906}
X_m(t)=2t-\farc{3}{2}\log t+x_m+o(1),~~\hbox{ as $t\to+\infty$}.
\end{equation}
Here, $x_m$ is a constant that depends on $m$ and the initial condition $v_0$. Bramson's original 
proof was probabilistic. A shorter probabilistic proof can be found in a recent paper~\cite{Roberts}, 
while the~PDE proofs can be found in~\cite{Lau,Uchiyama} and, more recently, in~\cite{HNRR13}. Various extensions to equations
with inhomogeneous coefficients 
have also been studied in~\cite{FZ2,FangZeitouni,HNRR12,MaillardZeitouni,NRR}.   
In this paper, we establish
a version of (\ref{aug2906}) -- but with the weaker $O(1)$ correction rather than $o(1)$ as in (\ref{aug2906}) --
for the solutions of the  non-local cane toads equation  (\ref{nov1006}).
We will assume that the initial condition is compactly supported on the right: there  exists $x_0$ such that $n_0(x) = 0$ for all $x \geq x_0$.
It has been shown in \cite{BouinCalvez} that (\ref{nov1006})-(\ref{nov1012}) admits a travelling wave solution
of the form $n(t,x,\theta)=\phi(x-c^*t,\theta)$. It is expected that the function $\phi(\xi,\theta)$ has the asymptotic decay 
\begin{equation}\label{aug2916}
\phi(\xi,\theta)\sim \xi e^{-\lambda^*\xi}Q(\theta),
\end{equation}
with a uniformly positive function $Q(\theta)>0$. While~\cite{BouinCalvez} does not show that travelling waves exist for all $c>c^*$, this is expected.  This would imply that $c^*$ is the minimal speed of propagation for the cane toads equation, in the same sense 
as $\tilde c^*=2$ is the minimal speed of propagation for the
Fisher-KPP equation (see also \cite[Remark 4]{BouinCalvez}). A precise characterization of the minimal speed $c^*$ and the 
decay rate $\lambda^*$ from~\cite{BouinCalvez} is recalled in Section~\ref{sec:spectral}.
Here is our main result. 
\begin{theorem}\label{p:toads}
Let $n(t,x,\theta)$ satisfy the system \eqref{nov1006}-(\ref{nov1012}), with the initial condition $n_0(x)\ge 0$
satisfying the assumptions above. There exists $m_0$ such that for all 
$\eps\in (0,m_0)$, there is a positive constant $C_\eps$ such that
\[\begin{split}
	&\liminf_{t\to\infty} \inf_{x \leq c^*t - \frac{3}{2\lambda^*}\log(t) - C_\eps} n(t,x) \geq m_0 - \eps,\\
	&\limsup_{t\to\infty} \sup_{x \geq c^*t - \frac{3}{2\lambda^*}\log(t) + C_\eps} n(t,x) \leq \eps.
\end{split}\]
\end{theorem}
The main difficulty in the proof of  Theorem~\ref{p:toads} is the lack of the maximum principle. 
In order to circumvent this, we obtain a present-time Harnack inequality for $n$, described below, which
is of an independent interest.  Using this, we reduce the problem 
to showing the logarithmic delay for the local Fisher-KPP system (\ref{nov1004}), a 
much simpler problem, as it obeys the maximum
principle. The analysis for the local equation
follows the general strategy of~\cite{HNRR12}, with some non-trivial modifications.

\subsubsection*{A parabolic Harnack inequality}

We will make use of the following version of the Harnack inequality, that 
is new, to the best of our knowledge.
Consider an operator
\begin{equation}\label{jan2602}
	Lu=\sum_{ij} a_{ij}(x)\farc{\partial^2u}{\partial x_i \partial x_j}.
\end{equation}
Here, $A(x) := \left( a_{ij}(x) \right)$ is a H\"older continuous, uniformly elliptic matrix:
there  exist~$\lambda>0$ and~$\Lambda>0$ such that 
\[
\forall x \in \R^n, \qquad \lambda I \leq A(x) \leq \Lambda I,
\]
in the sense of matrices.
\begin{theorem}\label{p:harnack}
Suppose that $u$ is a positive solution of
\begin{equation}\label{e:heat_equation}
	u_t - L u = 0,~t>0,~x\in\Rm^n.
\end{equation}
For any $t_0>0$, $R>0$, and  $p > 1$, 
there exists a constant $C$ such that if $t \geq t_0$ and~$|x - y| \leq R$, then
\begin{equation}\label{mar3004}
u(t,x) \leq C \|u_0\|_{\infty}^{1-1/p} u(t,y)^{1/p}.
\end{equation}
Moreover, $C$ depends only on $\lambda$, $\Lambda$, $n$, $t_0$, $R$, and $p$.
\end{theorem}
We point out that~\Cref{p:harnack} does
not hold with $p = 1$.  Indeed, when $n =1$ and $(a_{ij}) = I$, the solution $u(t,x) = t^{-1/2}\exp\left\{-x^2/4t\right\}$ does not satisfy~\eqref{mar3004}.

The paper is organized as follows. First, we prove  Theorem~\ref{p:harnack} in Section \ref{s:harnack}.
Then, in Section~\ref{sec:reduction}, we use the Harnack inequality to reduce the spreading rate question
for the non-local cane toads equation to that for the local problem (\ref{nov1004}). Section~\ref{sec:log-proof}
contains the   proof of the corresponding result for the local equation, with its most
technical part presented in Section~\ref{s:p_decay}. 

{\bf Acknowledgement.}   EB was supported by ``INRIA Programme Explorateur".  LR was supported by NSF grant DMS-1311903. 
Part of this work was performed within the framework of the LABEX MILYON (ANR- 10-LABX-0070) of Universit\'e de Lyon, 
within the program ÒInvestissements dÕAvenirÓ (ANR-11- IDEX-0007) operated by the French National Research Agency (ANR). 
In addition, CH has received funding from the European Research Council (ERC) under the European Unions Horizon 2020 research and 
innovation programme (grant agreement No 639638).

\section{A present-time parabolic Harnack inequality} \label{s:harnack}

In this section, we prove Theorem~\ref{p:harnack}. It is a consequence of 
a  small time heat kernel estimate due to Varadhan~\cite{Varadhan}. 
Let $G(t,x,y)$ be the fundamental solution to~\eqref{e:heat_equation}:  
\begin{equation}\label{e:heat_kernel}
\begin{cases}
	G_t =L_x G,~~t>0,~~x,y\in\Rm^n,\\
	G(0,\cdot,y) = \delta(\cdot - y),
\end{cases}
\end{equation}
so that the solution of
\[\begin{cases}
	u_t - Lu = 0,~~~t>0,~~x\in\Rm^n,\\
	u(0,x) = u_0(x),
\end{cases}\]
can be written, for all $t>0$ and $x\in \R^n$, as 
\[
	u(t,x) = \int_{\R^n} G(t,x,y) u_0(y) dy.
\]
The notation $L_x$ in (\ref{e:heat_kernel}) means that the operator $L$  acts on $G$ 
in the $x$ variable.
There are well-known Gaussian bounds for $G$ (see e.g.~\cite{Aronson, FabesStroock}) 
of the type
\[
\frac{c_1}{t^{n/2}}e^{-c_2 \frac{|x-y|^2}{t}}\le G(t,x,y)\le 
\frac{C_1}{t^{n/2}}e^{-C_2\frac{|x-y|^2}{t}},
\]
for $(t,x,y) \in \R^+ \times \R^n \times \R^n$. However, these are not precise enough in their dependence 
on $x$ and~$y$ for our purposes, as they do not control the constants
$c_2$ and $C_2$ very well. 

To state Varadhan's estimate, we introduce some notation.  
Given a matrix $A(x)=(a_{ij}(x))$, the associated Riemannian metric $d_A$ is
\[
d_A(x,y) = \inf_{\substack{p \in C^1([0,1]),\\ p(0) = x, p(1) = y}} 
\int_0^1 \sqrt{\dot p(\tau) A^{-1}(p(\tau)) \dot p (\tau)} d\tau.
\]
The ellipticity condition on the matrix $A$  
implies that $d_A$ and $|\cdot|$ yield equivalent metrics.
\begin{theorem}[Theorem~2.2~\cite{Varadhan}]
The limit
\[
	\lim_{t\to0} \left(- 4t \log G(t,x,y)\right) = d_A(x,y)^2
\]
holds uniformly for all $x$ and $y$ such that $|x - y|$ is bounded.
\end{theorem}

This agrees with the usual heat 
kernel when $L = \Delta$ since then $A = I$ 
and  $d_{A}(x,y)=|x - y|$.   
We may not use this result as 
stated as we will require a uniform estimate over all $x$ and $y$,
without a restriction to a compact set. 
However, it is easy to check 
that the proof in~\cite{Varadhan}, 
with a few straightforward modifications, implies the following.
\begin{theorem}\label{p:heat_kernel}
Given any $\eps >0$, 
the following inequalities hold uniformly over all $x,y\in\Rm^n$:
\begin{eqnarray}\label{031502}
&&\liminf_{t\to0}\big[ - 4t \log G(t,x,y)\big] 
\geq (1 - \eps) d_A(x,y)^2,\\
&&\limsup_{t\to0} \big[- 4t \log G(t,x,y) \big]
\leq (1+\eps) d_A(x,y)^2.\nonumber
\end{eqnarray}
\end{theorem}

We can now proceed with the proof of \Cref{p:harnack}.
\begin{proof}[{\bf Proof of Theorem~\ref{p:harnack}}]
Without loss of generality we may assume that $y=0$ and $|x|\le R$
in~(\ref{mar3004}). Let us take  $t_0>0$ and write, for all $t>t_0$ and $x\in\R^n$:
\[
u(t,x) = \int_{\R^n} G(t_0,x,y) u(t-t_0,y) dy.
\]
We have, using the
maximum principle, with some $s\in(0,1)$,  
to be specified later:  
\begin{equation}\label{jan2604}
\begin{split}
u(t,x)&= \int_{\R^n} G(t_0,x,z) u(t-t_0,z) dz \\
&\leq \|u(t-t_0,\cdot)\|_\infty^{1/q} 
\int_{\R^n} \left(u(t-t_0,z)  G(t_0,x,z)^{s p}\right)^{1/p} \left(G(t_0,x,z)^{(1-s)q}\right)^{1/q} dz\\
&\leq \|u_0\|_\infty^{1/q} 
\left( \int_{\R^n} u(t-t_0,z) G(t_0,x,z)^{sp} dz\right)^{1/p} 
\left(  \int_{\R^n} G(t_0,x,z)^{(1-s)q} dz\right)^{1/q}\\
&\leq C 
\|u_0\|_\infty^{1/q} 
\left( \int_{\R^n} u(t-t_0,z) G(t_0,x,z)^{sp} dz\right)^{1/p}.
\end{split}
\end{equation}
Here, we have chosen $q\in(1,\infty)$ satisfies
\[\frac 1{p}+\farc1{q}=1,
\]
and the constant $C>0$ depends on $t_0$ (in particular, it
blows up as $t_0\downarrow 0$).
The last inequality in (\ref{jan2604}) is an application 
of the bounds in \eqref{031502} since $s < 1$. 
Our next step is to show and use the following inequality:
there exist 
a constant~$C>0$ and $s>1/p$ that both depend on $t_0$, $R$, and $p$ such that
\begin{equation}\label{mar3002}
G(t_0,x,y)^{sp} \leq C G(t_0,0,y),
\end{equation}
for all $y\in\Rm^n$  and $|x|\le R$. 

Before proving \eqref{mar3002}, we shall conclude the proof of \Cref{p:harnack}. Using~\eqref{mar3002} in
\eqref{jan2604} gives
\begin{equation}\label{mar3006}
u(t,x)\leq C \|u_0\|_\infty^{1/q} \left( \int_{\R^n} u(t-t_0,y) G(t_0,0,y) dy\right)^{1/p}
= C \|u_0\|_\infty^{1/q} u(t,0)^{1/p},
\end{equation}
which is (\ref{mar3004}) with $y=0$.

To establish (\ref{mar3002}), we 
choose $s\in(0,1)$, $\eps>0$ and $\theta\in(0,1)$ such that  
\begin{equation}\label{e:rp}
sp(1 - \eps) > 1+\eps,
\end{equation}
and 
\[ 
1-\theta = \frac{(1+\eps)}{sp(1-\eps)}.
\]
We may now use Theorem~\ref{p:heat_kernel} to choose $t_0$ small enough so that
\begin{equation}\label{e:heat_kernel_application}
\begin{split}
	& - 4t_0 \log G(t_0,x,y) \geq (1 - \eps) d_A(x,y)^2 -  \eps,\\
	& - 4t_0 \log G(t_0,x,y) \leq (1+\eps) d_A(x,y)^2 + \eps,
\end{split}
\end{equation}
for all $x,y\in\Rm^n$.
%
Using \eqref{e:heat_kernel_application} and the triangle inequality
\[
d_A(x,y) \geq |d_A(x,0) - d_A(0,y)|,
\]
we get
\[
\begin{split}
\log [G(t_0,x,y)^{sp}] - \frac{sp\eps}{4t_0}
&\leq - sp (1-\eps) \frac{d_A(x,y)^2}{4t_0} \\
&\leq - sp (1-\eps) \frac{d_A(x,0)^2 - 2 d_A(x,0)d_A(y,0) + d_A(y,0)^2}{4t_0}.
\end{split}
\]
Young's inequality  
yields that
\[
\log [G(t_0,x,y)^{sp}] - \frac{sp\eps}{4t_0}
\leq \left(\frac{1}{\theta} - 1\right)
\frac{sp (1-\eps) d_A(x,0)^2}{4t_0} - 
\frac{sp(1-\eps)(1-\theta) d_A(y,0)^2}{4t_0}.
\]
Using the definition of $\theta$ and that the Euclidean metric and $d_A$ are equivalent, we deduce
\[ 
\log [G(t_0,x,y)^{sp}] - \frac{sp\eps}{4t_0}
\leq  \frac{CR^2}{t_0} - \frac{(1+\eps) d_A(y,0)^2}{4t_0},
\]
with a constant $C>0$ that depends on $\theta$, $p$ and $\eps$.
Applying the bounds in \eqref{e:heat_kernel_application} again, we obtain
\[
\log [G(t_0,x,y)^{sp}] - \frac{sp\eps}{4t_0}
		\leq \frac{CR^2}{t_0} + \log G(t_0,0,y) + \frac{\eps}{4t_0}.
\]
Exponentiating, we get (\ref{mar3002}),
finishing the proof.


\end{proof}

\section{A reduction to the local cane toads problem}\label{sec:reduction}

In this section, we show how to compare solutions of the non-local cane toads 
equation to the solutions of  a local
cane toads problem, of a more general form than (\ref{nov1004}).  
To do this, we use Theorem~\ref{p:harnack} to 
eliminate the non-local term in~(\ref{nov1006}).  
This will allow us to find two local cane toads equations   to which the 
solution of~(\ref{nov1006}) is a sub- and super-solution, respectively.

It has been shown in \cite{Turanova},  that solutions of  \eqref{nov1006} satisfy a uniform bound
\begin{equation}\label{aug3004}
n(t,x,\theta)\le M
\end{equation}
for all $(t,x,\theta)\in[0,\infty)\times\R\times\Theta$ with a constant $M$ depending only on $\underline\theta$ and $\overline \theta$.  With this in hand, we first show that we may bootstrap \Cref{p:harnack} to hold for $n$ as well.
\begin{proposition}\label{p:toads_harnack}
For any $t_0>0$, $R>0$, and $p >1$, 
there is a constant $C>0$  such that if
$t \geq t_0$ and $|\theta - \theta'| + |x-x'| \leq R$, and $n$ is a solution of~\eqref{nov1006}-\eqref{nov1012}, then
\begin{equation}\label{aug3006}
	n(t,x,\theta) \leq C n(t,x',\theta')^{1/p}.
\end{equation}
\end{proposition}

\begin{proof}[{\bf Proof of \Cref{p:toads_harnack}}]

The proof is by comparing $n$ to a solution to an associated linear heat equation. Take $t_1\ge t_0$
and let $h$ be the solution to  
\[
h_t = \theta h_{xx} + h_{\theta\theta},
\]
with the Neumann boundary conditions
\[
h_\theta(t,x,\underline \theta)=h_\theta(t,x,\overline\theta)=0,
\]
and the initial condition
\[ 
h(0,x,\theta) = n(t_1 - \delta, x,\theta),  
\]
with $\delta = \min\{1, t_0/2\}$. 
Theorem \ref{p:harnack} implies%
\footnote{Strictly speaking, to apply \Cref{p:harnack}, we need 
$n$ to be defined on $\R^2$, not on $\R\times\Theta$.  This obstacle, however, may be avoided considering a periodic extension of 
$n$ to $\R^2$; see~\cite[Section 2.1]{Turanova} for more details.} 
that there is a constant $C$ depending 
only on $M$, $\delta$,~$R$ and $p$ such that, for any $|x-x'| \leq R$ and $\theta \in [\underline\theta,\overline\theta]$, we have
\[
h(t,x,\theta) \leq C h(t,x',\theta)^{1/p},
\]
for all $t\ge\delta$.

On the other hand, as 
\[
n(1- M|\Theta|) \leq n(1 - \rho) \leq n,
\]
the comparison principle implies that 
\[
e^{(1-M|\Theta|)t} h(t,x)\le n(t_1-\delta+t,x)\le  e^{t}h(t,x).
\]
Hence, we may pull the Harnack inequality from $h$ to $n$: 
for all $(x,\theta) \in \R\times\Theta$ and $(x',\theta')\in \R\times \Theta$ such that $|x-x'|+|\theta-\theta'|\le R$ we have
\[
n(t_1,x,\theta) \leq e^\delta h(\delta,x,\theta)
\leq C e^\delta h(\delta,x',\theta')^{1/p}
\leq C e^{\delta} \left( e^{(M|\Theta|-1)\delta} n(t_1,x',\theta') \right)^{1/p}.
\]
This finishes the proof.

\end{proof}
  
We now construct two local cane toads
problems for which $n$ is a sub- and super-solution. 
We fix~$p \in (1, 3/2)$ and find $C>0$ so that we may apply Proposition~\ref{p:toads_harnack} 
with $t_0 = 1$ and $R = |\Theta|$, to obtain (after integration) 
\[
	\frac{n(t,x,\theta)^p}{C^p} \leq \rho(t,x) \leq C n(t,x,\theta)^{1/p}.
\]
for all $t\geq 1$, $x\in \R$ and $\theta\in\Theta$.  It follows that
\begin{equation}\label{aug3002}
	n\left(1 - C n^{1/p}\right)
		\leq n( 1- \rho)
		\leq n\left( 1 - \frac{n^p}{C^p}\right).
\end{equation}
This implies that for $t\ge 1$ the function $n(t,x,\theta)$ 
is a super-solution to the equation
\begin{equation}\label{e:toads_lower}
\underline u_t - \theta \underline u_{xx} - \underline u_{\theta\theta} =  \underline u ( 1 -C \underline u^{1/p}),  
\end{equation}
and a sub-solution to the equation
\begin{equation}\label{e:toads_upper}
	\overline u_t - \theta \overline u_{xx} -  \overline u_{\theta\theta} = \overline u \left( 1 - C^{-p} \overline u^p\right).
\end{equation}
Here, $\underline u$ and $\overline u$ satisfy 
the same Neumann boundary conditions (\ref{nov1012}) as $n$.  

We now choose the initial conditions at $t=0$:
$\underline u_0(x,\theta)=\underline u(0,x,\theta)$ and $\overline u_0(x,\theta)=\overline u(0,x,\theta)$,  
so that the ordering 
\begin{equation}\label{aug3016}
\underline u(t=1,x,\theta) \leq n(t=1,x,\theta) \leq \overline u(t=1,x,\theta)
\end{equation}
holds for all $x$ and $\theta$.  
This will guarantee that
\begin{equation}\label{e:ordering}
	\underline u(t,x,\theta) \leq n(t,x,\theta) \leq \overline u(t,x,\theta)
\end{equation}
for all $t \geq 1$ and all $x$ and $\theta$, because of (\ref{aug3002}). 
We only describe how $\underline u_0$ is chosen, but the process is similar for $\overline u_0$.  

To this end, let $h$ be a solution to the equation
\begin{equation}\label{aug3010}
h_t - \theta h_{xx} -  h_{\theta\theta} = 0,
\end{equation}
with the initial condition $h_0(x,\theta)=n_0(x,\theta)$.
Define the function $\underline h = e^{(1-M|\Theta|)t}h$, which satisfies
\begin{equation*}
\underline h_t = \theta \underline h_{xx} + \underline h_{\theta \theta} + \left(1 - M \vert \Theta \vert \right) \underline h,
\end{equation*}
where $M$ is the upper bound for $n$ from~\eqref{aug3004}.  Notice that $n$ is a super-solution to $\underline h$.
Hence
\begin{equation}\label{aug3014}
n(t=1,x,\theta) \geq \underline h(t=1,x,\theta)
= e^{\left(1 - M \vert \Theta \vert \right)} h(t=1,x,\theta),
\end{equation}
for all $x$ and $\theta$. 
On the other hand,  for any $a>0$, the function
\begin{equation}\label{eq:h_bars}
\overline h=a e^{t}h
\end{equation}
is
a super-solution for the equation for $\underline u$~(\ref{e:toads_lower}).  Hence, if $\underline u$ is the solution of (\ref{e:toads_lower}) with the
initial condition $\underline u_0 =an_0 $, then 
\begin{equation}\label{aug3012}
\underline u(t=1,x,\theta) \leq  a e h(t=1,x,\theta).
\end{equation}
Putting (\ref{aug3014}) and (\ref{aug3012}) together gives us
\[
\underline u(t=1,x,\theta) \leq a e^{M|\Theta|} n(1,x,\theta)
\]
for all $x$ and $\theta$. Thus, if we choose $a = \exp(- M|\Theta|)$ then the first inequality 
in (\ref{aug3016}) holds. Similarly, we may choose an initial condition $\bar u_0$ 
so that the second inequality in (\ref{aug3016}) holds as well.

\section{The logarithmic correction in the local cane toads fronts}\label{sec:log-proof}


We have shown that there exist functions $\underline{u}$ 
and $\overline u$, satisfying the local cane toads  
equations (\ref{e:toads_lower}) and~(\ref{e:toads_upper}), respectively, such that 
the solution $n$ of (\ref{nov1006})-(\ref{nov1012}) satisfies the lower and upper bounds 
in (\ref{e:ordering}). Therefore,  Theorem~\ref{p:toads} is
a consequence of the corresponding result for the Fisher-KPP equations. We present the local
Fisher-KPP result in a slightly greater generality than what is needed for Theorem~\ref{p:toads},
as the extra generality introduces no extra complications in the proof.

%
Let $D$ be a uniformly positive and bounded function on a smooth domain $\Theta \subset \R^d$, and let $A$ be a   
$C^1$ function on $\Theta$.  Let  
$u$ be the solution to the Fisher-KPP equation
\begin{equation}\label{e:local_toads}
u_t - D u_{xx} -  \Delta_\theta u + A u_x = f(u), 
\end{equation}
with the Neumann boundary conditions:
\begin{equation}\label{aug3020}		
\pdr{u}{\nu_\theta} (t,x, \theta) =  0,
 \end{equation}
and the initial condition $u(0,\cdot) = u_0$. Here, $\nu_\theta$ is the normal to $\partial\Theta$.
We assume that
\begin{eqnarray}\label{aug3026}
 \liminf_{x\to-\infty} u_0(x,\theta) > 0,
\end{eqnarray}
uniformly in $\theta\in\Theta$, that $u_0 \geq 0$, and that there is some $x_0$ such that $u_0(x,\theta) = 0$ for all $x \geq x_0$.
The nonlinearity $f$ is of the Fisher-KPP type:
there exist $u_m>0$, $M>0$ and $\delta > 2/3$ such that 
\begin{equation}\label{aug3028}
\hbox{$f(0) = f(u_m) = 0$, \qquad $f(u) > 0$ \qquad for all~$u \in [0,u_m]$,}
\end{equation}
and
\begin{equation}\label{e:f}
u - M_\delta u^{1+ \delta} \leq f(u) \leq u, ~~\text{ for all } u \in [0,u_m].
\end{equation}

A classical result of Berestycki and Nirenberg~\cite{BerestyckiNirenberg} shows that (\ref{e:local_toads}) admits travelling wave solutions
of the form $u(t,x,\theta)=\Phi(x-ct,\theta)$, with $\Phi(x,\theta)$ such that
\begin{equation}\label{aug3022}
-c\Phi_x- D \Phi_{xx} -  \Delta_\theta\Phi  + A \Phi_x = f(\Phi), ~~
\end{equation}
and $\Phi(-\infty,\cdot)=u_m$, and $\Phi(+\infty,\cdot)=0$. In addition, $\Phi$ satisfies the Neumann boundary conditions~(\ref{aug3020}), and
$0<\Phi(x,\theta)<u_m$ for all $x$ and $\theta$. Such travelling waves exist for all $c\ge c^*$, with the same $c^*$ as in Theorem~\ref{p:toads},
and the travelling wave corresponding to the minimal speed has the asymptotics
\[
\Phi(\xi,\theta)\sim \alpha \xi e^{-\lambda^*\xi} Q(\theta),~~\hbox{ as $x\to+\infty$},
\]
with the same exponential decay rate $\lambda^*$ and profile $Q$ as in \eqref{aug2916}. 
Once again, a precise description of $c^*$ and $\lambda^*$ in terms of an eigenvalue problem will be given in 
Section~\ref{sec:spectral}. What is important for us is that, as far as the function $f$ is concerned,
both $c^*$ and $\lambda^*$ depend only on $f'(0)$ but not, say,  on $u_m$ or $\delta$.

By translating and scaling and by changing to a constant speed moving reference frame, if necessary, we may assume without loss of generality that $u_m = 1$, $f'(0)=1$, 
that the drift $A$ has mean-zero, and, finally, that the initial condition~$u_0$ is not identically equal to zero on the half-cylinder $\{x>0,~\theta\in\Theta\}$.

\begin{theorem}\label{p:kpp}
Suppose that $D$ and $A$ are as above and $f$  satisfies \eqref{aug3028}-\eqref{e:f}. 
There exist $c^*>0$ and $\lambda^*>0$  that, as far as  $f$ is concerned,
depend only on~$f'(0)$, with the following
property. Let  $u$ satisfy \eqref{e:local_toads}-\eqref{aug3020}, with the initial condition
$u_0$ as above~\eqref{aug3026}.
Then, for any $m \in (0,u_m)$, there exist $x_m>0$ and $T_m>0$, depending on $m$, such that if $t \geq T_m$ we have
\begin{equation}\label{aug3102}
\{x \in \R \, : \, \exists \theta \in \Theta, \,u(t,x,\theta) = m\}
\subset \left[ c^* t - \frac{3}{2\lambda^*}\log(t) - x_m,
c^* t - \frac{3}{2\lambda^*}\log(t) + x_m\right].
\end{equation}
\end{theorem}

Theorem~\ref{p:toads} follows from Theorem~\ref{p:kpp} and the bounds on $n$ 
in (\ref{e:ordering}), in terms of the solutions of the Fisher-KPP equations  (\ref{e:toads_lower}) and (\ref{e:toads_upper}).
The reason is that  $c^*$ and $\lambda^*$ for the two non-linearities in  (\ref{e:toads_lower}) and (\ref{e:toads_upper}) coincide,
hence the level sets of the corresponding solutions $\underline u$ and $\overline u$ of these two equations
 stay within $O(1)$ from each other, and (\ref{e:ordering}) means that so do the level sets of the solution of (\ref{nov1006}).
%


%

The proof of Theorem \ref{p:kpp}
mostly follows the strategy of \cite{HNRR12} where a similar result has been proved
in the one-dimensional periodic case.  A general multi-dimensional form of the Bramson
shift is a delicate problem~\cite{Shabani}. However, the particular 
form of the  present problem allows us to streamline many of the 
details and modifies some of the steps in the proof.  
Typically, the spreading speed  $c_*$
of the solutions of the Fisher-KPP type equations can be inferred from the linearized problem,
that in the present case takes the form 
\begin{equation}\label{mar1802}
u_t+A u_x=D u_{xx}+ \Delta_\theta u+f'(0)u.
\end{equation}
The main qualitative difference between the solutions of (\ref{mar1802}) and those of the nonlinear Fisher-KPP
problem is that the former grow exponentially in time on any given compact set, while the latter remain bounded.
A remedy for that discrepancy is to consider (\ref{mar1802}) in a domain with a moving boundary:
$x>X(t)$, with 
\begin{equation}\label{aug3104}
X(t)=c^*t-r(t),
\end{equation}
with the Dirichlet boundary condition $u(t,X(t),\theta)=0$. Then the shift $r(t)$ is chosen so that the solutions of the moving boundary
problem remain $O(1)$ as $t\to+\infty$. It turns out that such ``correct" shift is exactly 
\begin{equation}\label{aug3106}
r(t)=\farc{3}{2\lambda^*}\log t,
\end{equation}
as in (\ref{aug3102}). 
This allows to use them as sub- and super-solutions to the nonlinear Fisher-KPP equation, to
prove that the front of the solutions to (\ref{e:local_toads}) is also located at a distance $O(1)$ from~$X(t)$ given by (\ref{aug3104})-(\ref{aug3106}), which is the claim
of Theorem~\ref{p:kpp}.

\subsection{The eigenvalue problem defining $c^*$ and $\lambda^*$.}\label{sec:spectral}

Let us first recall from~\cite{BerestyckiNirenberg} how $c^*$ and $\lambda^*$ are defined in Theorems~\ref{p:toads} and \ref{p:kpp}.
We look for exponential solutions of the linearized cane toads equation (\ref{mar1802}), with $f'(0)=1$, of the form
 \begin{equation}\label{mar1804}
u(t,x,\theta)=e^{-\lambda(x-ct)}Q_\lambda(\theta).
\end{equation}
This leads to
the following spectral problem  
on the cross-section~$\Theta$ for the unique
positive
eigenfunction $Q_\lambda>0$:
\begin{equation}\label{mar1806}
\left\{ 
\begin{array}{ll}
\Delta_\theta Q_\lambda + ( \lambda^2 D  + 
\lambda A  -\lambda c(\lambda) + 1)Q_\lambda(\theta)= 0,& \qquad  \hbox{ in $\Theta$}, \medskip\\
\pdr{Q_\lambda}{\nu_\theta} =    0,& \qquad \hbox{ on $\partial\Theta$.}\nonumber
\end{array}
\right.
\end{equation}
We will use the normalization
\begin{equation}\label{mar1808}
\int_\Theta Q_\lambda(\theta) d\theta = 1.
\end{equation}
In other words, given $\lambda>0$, 
we solve the eigenvalue problem
\begin{equation}\label{mar1810}
\left\{ 
\begin{array}{ll}
\Delta_\theta Q_\lambda + ( \lambda^2 D  + 
\lambda A  )Q_\lambda(\theta)= \mu(\lambda)Q_\lambda,& \qquad  \hbox{ in $\Theta$,} \medskip\\
\pdr{Q_\lambda}{\nu_\theta} =    0,& \qquad  \hbox{ on $\partial\Theta$.}
\end{array}
\right.
\end{equation}
It has a unique positive eigenfunction $Q_\lambda$ corresponding to its principal
eigenvalue $\mu(\lambda)$ -- this is a standard consequence of the Krein-Rutman
theorem.  The positivity of $\mu(\lambda)$ easily follows by dividing~\eqref{mar1810} 
by $Q_\lambda$, integrating, and using the positivity  of $Q_\lambda$ and the boundary conditions,  
along with the normalization 
\[
\int_\Theta A(\theta)d\theta = 0.
\]  
Then, the speed $c(\lambda)$ is determined by
\begin{equation}\label{eq:mu}
\mu(\lambda)=\lambda c(\lambda)-1,
\end{equation}
that is,
\begin{equation}\label{mar1812}
c(\lambda)=\farc{1+\mu(\lambda)}{\lambda}.
\end{equation}
We will use the notation, well-defined by the following proposition,
\begin{equation}\label{mar1814}
c^*=\min_{\lambda>0}c(\lambda),~~
\lambda^* = \underset{\lambda > 0}{\text{argmin} } \, c(\lambda),
\end{equation}
and denote by $Q^*$ the corresponding eigenfunction.
\begin{prop}\label{propspec}
The function $\lambda \mapsto c(\lambda)$ has a minimum $c^*$, and
%
%
%
%
%
\begin{equation}\label{rel3}
c^* {\int_{\Theta} \left( Q^* \right)^2 d \theta}= 
 {\int_{\Theta} [2 \lambda^*D(\theta)+A(\theta)] \left( Q^* \right)^2 d \theta}.
\end{equation}
 Further, we have $c''(\lambda^*) > 0$.
\end{prop}
{\bf Proof of Proposition \ref{propspec}.}
%
%
Since $Q_\lambda \in{C}^2(\Theta)$ and satisfies 
Neumann boundary conditions, there exists $\theta_0$ such that $\Delta Q_{\lambda}(\theta_0) = 0$. 
%
%
We deduce from (\ref{mar1806}):
\begin{equation*}
c(\lambda) = \frac{1}{\lambda} + A(\theta_0) + \lambda D(\theta_0).
\end{equation*}
As the functions $A(\theta)$ and $D(\theta)$ are bounded, and $D(\theta)$
is uniformly positive, $c(\lambda)$ satisfies
\begin{equation*}
c(\lambda) \underset{\lambda \to 0}{\sim} \frac{1}{\lambda}, \qquad \lambda c(\lambda) = \mathcal{O}_{\lambda \to + \infty}(\lambda^2).
\end{equation*}
The continuity of the function $c(\lambda)$ implies 
the existence of a positive minimal speed $c^*$ and 
a smallest positive minimizer $\lambda^*$.

Differentiating \eqref{mar1806} with respect to $\lambda$, we obtain
\begin{equation*}
\left( - \lambda c'(\lambda) - c(\lambda) + 
A + 2 \lambda D \right) Q_\lambda 
+ \left(  \lambda^2 D + \lambda A  -\lambda c(\lambda) + 1 \right) \frac{\partial Q_\lambda}{\partial \lambda} + 
\Delta_{\theta } \left( \frac{\partial Q_\lambda}{\partial \lambda} \right) = 0.
\end{equation*}
Let us multiply by $Q_\lambda$ and integrate. We obtain, for all $\lambda >0$, 
\begin{equation}\label{rel4}
\int_\Theta \left( - \lambda c'(\lambda) - c(\lambda) + A(\theta) + 
2 \lambda D(\theta) \right) Q_\lambda^2  d\theta = 0.
\end{equation}
In particular, for $\lambda = \lambda^*$, we have $c'(\lambda^*)=0$, and
(\ref{rel3}) follows.  Finally, for the last claim, it is easy to see by differentiating twice \eqref{eq:mu} and using $c'(\lambda^*)=0$ that 
\[
c''(\lambda^*)=\frac{\mu''(\lambda^*)}{\lambda^*}.
\]
In addition, the variational principle for 
the principal eigenvalue $\mu(\lambda)$ of (\ref{mar1810})   
implies that $\mu(\lambda)$ is a convex function. A straightforward
computation shows that actually $\mu''(\lambda^*)>0$, thus $c''(\lambda^*)>0$.

%

\subsection{A ``heat equation'' bound for the local cane toads equation}

Motivated by the exponential solutions, we may decompose 
a general solution $u(t,x,\theta)$ of the linearized Fisher-KPP equation 
(\ref{mar1802}) as
\begin{equation}\label{mar2902}
u(t,x,\theta)=e^{-\lambda^*(x-c^*t)}Q^*(\theta)p(t,x,\theta).
\end{equation}
The function $p(t,x,\theta)$ then satisfies
\begin{equation}\label{mar2904}
p_t = D p_{xx} + \Delta_\theta p
- \left( 2\lambda^* D + A  \right) p_x + \frac{2 }{Q^*} \nabla_\theta Q^*_\theta\cdot\nabla_\theta p,
\end{equation}
with the Neumann boundary conditions
\begin{equation}\label{aug3036}
\pdr{p}{\nu_\theta}=0, \qquad \hbox{ on $\partial\Theta$.}
\end{equation} 
If $D \equiv 1$ and $A \equiv 0$, then $Q^*\equiv 1$ and
$c^*=2\lambda^*$, meaning that (\ref{mar2904}) is simply the standard heat 
equation in the frame moving with speed $c^*$. As we have mentioned, in order to keep the solutions of the
linearized problem bounded, we need to impose the Dirichlet boundary condition
at a moving boundary.
The next proposition shows that, in general,
the special form of the drift terms in~(\ref{mar2904}) balances exactly
so that the solutions decay as those of the heat equation,
with the Dirichlet boundary condition imposed. We formulate it for a slightly
more general equation than~(\ref{mar2904}), which we will need below.
\begin{prop}\label{p:p_decay}
Let $\omega : \R^+ \mapsto \R^+$, $\overline \omega$, $C$, and $T$ be such that 
\begin{equation}\label{sep702}
\tau\omega(\tau)  \to\overline\omega \hbox{ as }\tau\to+\infty, \qquad\qquad
|\omega'(\tau) \tau^2|, (\tau+T)\,\omega(\tau) \leq C
\end{equation}
and let $p_0$ be 
a non-zero, non-negative function such that   
that $p_0(x) = 0$ for all $x > x_0$ and such that $\1_{[0,\infty)} p_0$ is non-zero.  Suppose that 
$p$ satisfies
\begin{equation}\label{e:p_equation}
(1 - \omega)p_\tau = D  p_{xx} + \Delta_\theta p
- \left( 2\lambda^* D  + A  \right)p_x + \frac{2}{Q^*} \nabla_\theta Q^*\cdot\nabla_\theta p ,~~
\end{equation}
for $\tau>0$, $x>c^*\tau$, and $\theta\in\Theta$, with the Neumann boundary condition (\ref{aug3036}), 
the Dirichlet boundary condition for $\tau >0$, 
\begin{equation}\label{mar1820}
p(\tau, c^*\tau,\cdot) = 0,
\end{equation}
and the initial condition
$p(0,\cdot) = p_0$.
There exists $T_0$ such that if $T\geq T_0$, then there exist~$\sigma>0$ and~$C>0$  that do not depend on $p_0$, and~$\tau_0>0$ that may depend on $p_0$ such that
\begin{equation}\label{mar1822}
\frac{x-c^*\tau}{C \tau^{3/2}} \leq p(\tau, x,\theta) 
\leq \frac{C(x-c^*\tau)}{\tau^{3/2}},
\end{equation}
for all $x \in [c^* \tau, c^* \tau + \sigma \sqrt{\tau}]$, all $\theta \in \Theta$ and all $\tau \geq \tau_0$.
\end{prop}
As the proof is rather technical, we postpone it for the moment.  Its proof is in \Cref{s:p_decay}.

\subsection{The upper bound}

We will now show how to deduce the statement of
Theorem~\ref{p:kpp} from Proposition~\ref{p:p_decay}, starting with the upper bound. We will thus prove that the delay is at least $\frac{3}{2\lambda^*}\log(t)$ in the following sense: 
\begin{equation*}
\max \{x \in \R \, : \, \exists \theta \in \Theta, \,u(t,x,\theta) = m\}
\leq c^* t - \frac{3}{2\lambda^*}\log(t) + x_m,
\end{equation*}
for some constant $x_m$. The idea is to use the linearized problem with 
a moving Dirichlet boundary condition to create a suitable
super-solution.  Obviously, the Dirichlet boundary condition prevents the
solution of this problem from being directly 
a super-solution.  To overcome this, we show that the solution to the linearized equation is greater than $1$ near the moving boundary.  
Hence, after a suitable cut-off, it will be a true super-solution.

To this end, we consider the solution to the linearized
problem with the Dirichlet boundary condition at $x=c^*t-r\log (1+t/T)$, with
$r$ and $T$ to be determined:
\begin{equation}\label{sep2812}
\left\{
\begin{array}{ll}
z_t - D z_{xx} -  \Delta_\theta z + A z_x = z,&
\hbox{ for } x>c^*t - r \log(1 + t/T),\medskip \\
z(t,c^*t - r \log(1 + t/T), \cdot) = 0, &\medskip \\
\pdr{z}{\nu_\theta}=   0,&\hbox{ on } \partial\Theta,\medskip \\
z(0,\cdot) = u_0.
\end{array}
\right.
\end{equation}
We make a time change
\begin{equation}\label{sep2814}
\tau = t - \frac{r}{c^*} \log\left(1 + \frac{t}{T}\right). 
\end{equation}
By fixing $T$ large enough, depending only on $r$ and $c^*$, we may ensure
that the function $h(\tau) = t$ is one-to-one, and
\begin{equation}\label{sep2816}
\frac{1}{h'(\tau)}
	= 1 - \frac{r}{c^*(t + T)}
	= 1 - \frac{r}{c^*(\tau + T) + r\log(1 + t/T)}
	= 1 - \frac{r}{c^*(\tau + T)} + O(\tau^{-3/2}).
\end{equation}
To simplify the notation, we define 
\begin{equation}\label{aug3110}
\omega(\tau) = 1 - \frac{1}{h'(\tau)},~~|\omega'(\tau)| = O(\tau^{-2}).
\end{equation}
Notice that $\omega$ satisfies~\eqref{sep702}. The function  
$\tilde z(\tau,\cdot) = z(t,\cdot)$  satisfies
\[
(1 - \omega)\tilde z_\tau = \tilde D\tilde z_{xx} + 
 \Delta_\theta\tilde z  - A \tilde z_x + \tilde z.
\]
Let $\tau \mapsto \alpha(\tau)$ be a function to be determined later, 
and decompose $\tilde z$ as
\[
\tilde z(\tau,x,\theta) = \alpha(\tau) e^{-\lambda^*(x-c^* \tau)} 
Q^*(\theta) \tilde p(\tau,x,\theta).
\]
The function $\tilde p$ satisfies
\begin{equation}\label{mar3104}
(1 - \omega)\tilde p_\tau= D \tilde p_{xx}
+  \Delta_\theta\tilde p - (A+2\lambda D)p_x
+ \farc{2}{Q^*} \nabla_\theta Q^* \cdot\nabla_\theta\tilde p
+ \left( - \frac{\alpha'}{h' \alpha} + \frac{r\lambda^*}{t+T}\right)
\tilde p,
\end{equation}
and $\tilde p(\tau,c^*\tau, \cdot) = 0$ for all $\tau$. 
We choose $\alpha$ as the solution of
\begin{equation}\label{mar3102}
\frac{\alpha'}{\alpha} = \farc{r\lambda^*}{t+T}h'=
\frac{r\lambda^*}{\tau + T} + O\left(\frac{1}{(\tau + T)^{2}}\right),
\end{equation}
with the asymptotics:
\begin{equation}\label{e:alpha}
\alpha(\tau) = \exp\left\{ r \lambda^*\log(\tau + T) + 
O(\tau^{-1})\right\} = (\tau + T)^{r\lambda^*}(1 + O(\tau^{-1})).
\end{equation}
In view of \eqref{aug3110}, we may apply Proposition \ref{p:p_decay} to the solutions of (\ref{mar3104}).  
This, along with \eqref{e:alpha}, implies that 
if we choose
\begin{equation}\label{mar3018}
r=\farc{3}{2\lambda^*},
\end{equation}
then there exist constants $\sigma$, 
$C_1$ and $C_2$ and a fixed time $\tau_0$ such that we have 
\begin{equation}\label{mar3020}
C_1 (x-c^*\tau) e^{-\lambda^*(x-c^*\tau)}
\leq \tilde z(\tau,x,\theta)
\leq C_2 (x-c^*\tau) e^{-\lambda^*(x-c^*\tau)},
\end{equation}
for $\tau \geq \tau_0$ and all 
$x \in[c^*\tau, c^*\tau + \sigma\sqrt\tau]$.
Hence, we may choose $M$ such that 
\[
M\tilde z(\tau, c^*\tau + 1,\theta) 
\geq 2,
\]
for all~$\tau \geq \tau_0$ and~$\theta\in\Theta$.

We may now define a super-solution for the nonlinear Fisher-KPP equation  (\ref{e:local_toads}) as
 \[
\overline u (t,x,\theta) = 
\begin{cases}
\min(1, Mz(t,x,\theta)), ~~~~~~&\text{ for all } x \geq c^* t - r\log(1 + t/T) + 1, \medskip \\
1, & \text{ for all } x \leq c^*t - r\log(1 + t/T) +1.
\end{cases}
\]
Figure~\ref{figure1} depicts a sketch of the solution $u$ of the nonlinear Fisher-KPP problem,
and the super-solution $\overline u$.
We also have $\overline u(h(\tau_0),\cdot) \geq u_0$ for a sufficiently
large $M$, since $u_0$ is compactly supported on the right.  Hence, we have 
\[
u(t,\cdot)\leq \overline u (h(\tau_0) + t,\cdot)
\]
for all $t\ge t_0$.

To conclude, it follows from the form of our super-solution and (\ref{mar3020})
that, given any $m\in (0,1)$, we may choose $x_m\geq 1$ 
such that $\overline u(t,x,\theta)<m$ for all $t\ge t_0$, all
\[
x \geq c^* t - \frac{3}{2\lambda^*} \log t + x_m,
\]
and all $\theta\in\Theta$.
Thus, for such $x$ we have
\[
u(t,x,\theta) \leq \bar u(t,x,\theta)\le m,
\]
for all $t\ge t_0$ and $\theta\in\Theta$.
This concludes the proof of the upper bound in Theorem~\ref{p:kpp}.

\begin{figure} 
\begin{center}
\includegraphics[scale=.95]{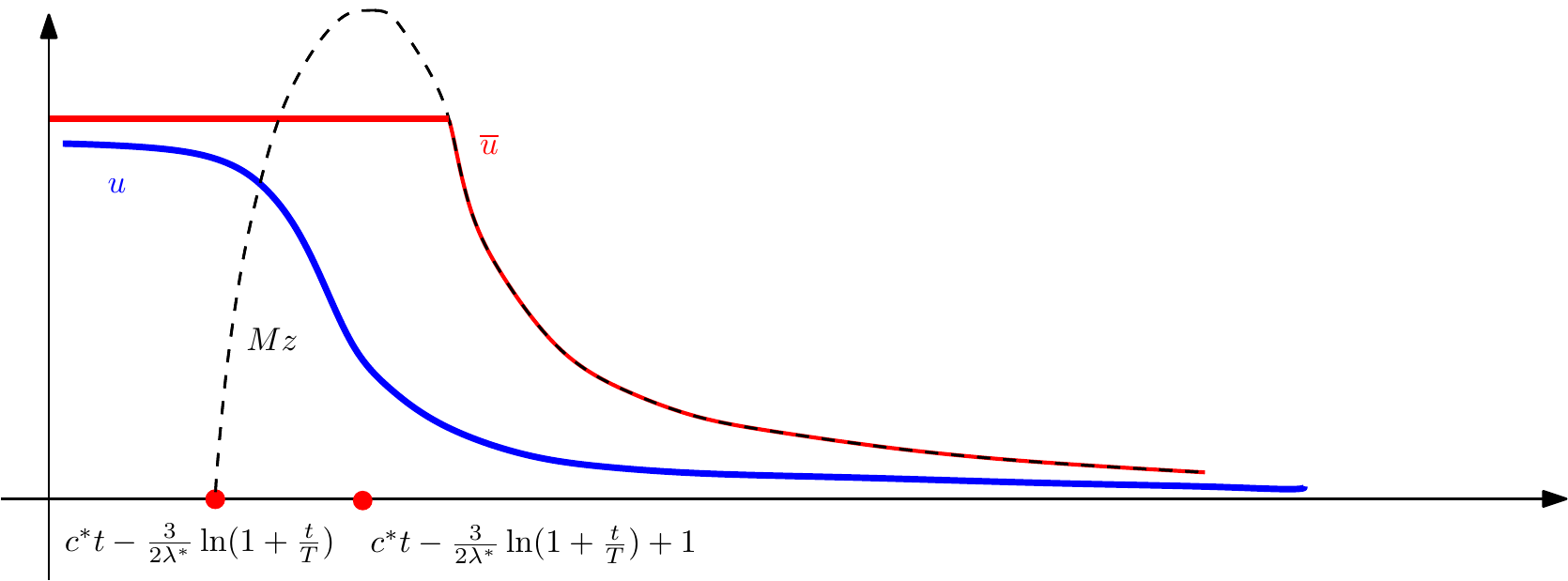}
\end{center}
\caption{A sketch of the solution $u$ and the super-solution $\overline u$.}
\label{figure1}
\end{figure}

\subsection{The lower bound}

We now prove that the delay is at most $\frac{3}{2\lambda^*}\log(t)$ in the following sense: 
\begin{equation*}
\min \{x \in \R \, : \, \exists \theta \in \Theta, \,u(t,x,\theta) = m\}
\geq c^* t - \frac{3}{2\lambda^*}\log(t) + x_m,
\end{equation*}
for some constant $C_m$.
The proof of the lower bound requires the same estimates as the upper bound, 
but the approach is slightly different.  Note that the solution to the 
linearized equation is not a sub-solution to the nonlinear Fisher-KPP equation
since $f(u)\le u$.  To get around this, we solve the linearized equation 
with a moving Dirichlet boundary condition at $c^* t$, 
instead of~$c^*t - (3/2\lambda^*)\log(t)$, 
in order to make this solution small.  Then, we modify the solution to the 
linearized equation by an order $O(1)$
multiplicative factor in order to obtain a 
sub-solution.

The resulting sub-solution will decay in time.  
Hence, we may not directly conclude a lower bound on 
the location of the level sets.  Instead, we show that 
this sub-solution is of the correct 
order $e^{-\sigma \sqrt{t}}/t$ at the position
$c^*t + \sigma\sqrt t$.  This will allow us to fit a travelling wave underneath the solution $u$ of the Fisher-KPP
equation on the half-line $x<c^*t + \sigma\sqrt t$, 
and we use this travelling wave to obtain a lower bound on the location 
of the level sets~of~$u$. We will assume without loss of generality that
\begin{equation}\label{aug3126}
\ell:=\liminf_{x\to-\infty}\inf_{\theta\in\Theta}u_0(x,\theta)=1.
\end{equation}
It is straightforward to modify the argument below to account for the case $\ell<1$. Note that $\ell>0$ by
assumption (\ref{aug3026}). As a consequence of (\ref{aug3122}) we have that, for all $t\geq 0$,
\begin{equation}\label{aug3128}
\liminf_{x\to-\infty}\inf_{\theta\in\Theta}u(t,x,\theta)=1.
\end{equation}

\subsubsection*{A preliminary sub-solution using the linearized system}

As outlined above, the first step  is to obtain a sub-solution  decaying in time.  To this end, we look at the solution $w$ to
\begin{equation}\label{mar3024}
\left\{
\begin{array}{ll}
w_t - Dw_{xx} -  \Delta_\theta w + A w_x = w,&
\hbox{ for } x>c^*t,\medskip \\
w(t,c^*t, \cdot) = 0, &\medskip \\
\pdr{w}{\nu_\theta}=   0,&\hbox{ on } \partial\Theta,\medskip \\
w(0,\cdot) = u_0.
\end{array}
\right.
\end{equation}
As before, we factor out  a decaying exponential, and
the eigenfunction $Q^*$:  
\begin{equation}\label{aug3118}
w(t,x,\theta) = e^{-\lambda^*(x-c^*t)} Q^*(\theta) p(t,x,\theta).
\end{equation}
The function $p$ satisfies
\begin{eqnarray}\label{mar3024bis}
&&p_t = D p_{xx} +  \Delta_\theta p  - 
\left( 2\lambda^* D + A  \right) p_x + 
\frac{2}{Q^*} \nabla_\theta Q^*\cdot\nabla_\theta p_\theta,~~~\text{ for $x>c^*t$, } 
\end{eqnarray}
with the corresponding boundary and initial conditions.
\Cref{p:p_decay} with~$\omega = 0$ gives an upper bound
\[
|p(t,x+c^*t)| \leq \frac{Cx}{(t+1)^{3/2}},
\]
%
%
that, along with the decomposition (\ref{aug3118}) gives  
\begin{equation}\label{eq:t32_decay}
\|w(t,\cdot,\cdot)\|_\infty
\leq \frac{C}{(1+t)^{3/2}}
\end{equation}
This temporal decay allows us to devise a sub-solution of the Fisher-KPP problem,
of the form  
\[
\underline{w}(t,x,\theta)
=a(t) w(t,x,\theta).
\]
To verify that $\underline{w}$ is a sub-solution, we note that
\[
\underline{w}_t - D\underline{w}_{xx} -  \Delta_\theta\underline{w} 
+ A \underline{w}_x - f(\underline w)
\leq \frac{\dot a(t)}{a(t)} \underline{w} + \underline{w} - 
(\underline{w} - M_\delta \underline{w}^{1 + \delta}),
\]
with $\delta$ as in (\ref{e:f}).
Using~\eqref{eq:t32_decay}, we get
\[
\underline{w}_t - D\underline{w}_{xx} -  \underline{w}_{\theta\theta} + 
A \underline{w}_x - f(\underline w)
\leq \underline{w}   
\left(\frac{\dot a(t)}{a} + \frac{CM_\delta }{(t+1)^{{3\delta}/{2}}}\right).
\]
We let $a(t)$ be the solution of
\begin{equation}\label{e:modulation}
- \frac{\dot a}{a} = \frac{CM_\delta}{(t+1)^{3\delta/2}}.
\end{equation}
As $\delta>2/3$, there exists $a_0>0$ so that $a(t)>a_0$ for all $t>0$.
Taking $a(0)\leq 1$ ensures that 
\[
\underline{w}(0,\cdot) \leq u_0(x,\cdot),
\]
while \eqref{e:modulation} implies
\[
\underline{w}_t - D\underline{w}_{xx} - 
 \underline{w}_{\theta\theta} + A \underline{w}_x - f(\underline w)
\leq 0.
\]
As a result, the maximum principle implies that
\[
	\underline{w}(t,c^*t + x, \theta) \leq u(t, c^*t + x, \theta),
\] 
for all $\theta$, all $t$ and all $x \geq 0$.
%
In particular, the conclusion of Proposition~\ref{p:p_decay} implies that there exists~$\sigma > 0$ and $T_0$ such that if $t \geq T_0$ then
\begin{equation}\label{e:u_lower_bound}
	\frac{C a_0 e^{-\sigma\sqrt{t}}}{t}
		\leq a_0 w(t,c^*t + \sigma \sqrt t,\theta)
		\leq u(t,c^*t + \sigma \sqrt t, \theta).
\end{equation}

\subsubsection*{A travelling wave sub-solution}

We now use the lower bound (\ref{e:u_lower_bound}) to fit a travelling wave under $u$.  
The sub-solution we will construct is sketched in Figure~\ref{figure2}.
In order to avoid complications due to boundary conditions 
at~$-\infty$, we fix $\overline m$ to be any constant in $(m,1)$, and
replace the non-linearity $f(u)$ by $f(u)(1 - u/\overline m)$.
Let~$U$ be the
travelling wave solution to the modified equation moving with 
speed $c^*$:
\begin{equation}\label{aug3120}
-c^*U_x - DU_{xx} - \Delta_\theta U+ 
A U_x - f(U )(1 - U /\overline m) = 0,
\end{equation}
with the Neumann boundary condition at $\partial\Theta$, and
\begin{equation}\label{aug3122}
U(-\infty,\cdot)=\overline m,~~U(+\infty,\cdot)=0.
\end{equation}
This wave satisfies $0 < U < \overline m$,  
so it sits below $u$ 
as $x$ tends to $-\infty$: see (\ref{aug3128}). 
However, it moves too quickly -- it does not have the logarithmic
delay in time.  Instead, we define
\begin{equation}\label{aug3144}
\underline U(t,x,\theta) = U (x - c^*t + s(t), \theta).
\end{equation}
It is easy to check that if $\dot s(t)\ge 0$, then $\underline U$ is a sub-solution to (\ref{aug3120}):
\begin{multline}\label{aug3134}
\underline U_t - D \underline U_{xx} -  \Delta_\theta\underline U + A \underline U_x - f(\underline U)(1- \underline U/\overline m )
\\ = -(c^* - \dot s(t))U_x - DU_{xx} -  \Delta_\theta U+ AU_x + f(U)(1 -U/{\overline m})=  \dot s(t) U_x\le 0,
\end{multline}
as $U$ is decreasing in $x$~\cite{BerestyckiNirenberg}. Hence, 
$\underline U$ is a sub-solution. 

We already know from (\ref{aug3128})  that $\underline U$ sits below $u$ at $x=-\infty$:
\begin{equation}\label{aug3136}
\underline U(t,x,\theta)<u(t,x,\theta),\hbox{ for all $t>0$ and $\theta\in\Theta$ for all $x$ sufficiently negative}.
\end{equation}
Thus, we only need to 
arrange for $\underline U$ to sit below~$u$ 
at $x = c^*t + \sigma\sqrt t$, 
with~$\sigma$ is as in (\ref{e:u_lower_bound}).  The 
travelling wave has the asymptotics~\cite{Hamel}
\begin{equation}\label{e:tw_asymptotics}
U(x,\theta)  \sim  xe^{-\lambda^*x}Q^*(\theta) 
\end{equation}
for large $x$ (uniformly in $\theta$).  By translation, 
we may ensure that
\[
U (x,\theta) \leq \eps x e^{-\lambda^* x},
\]
for all $x \geq 1$, with $\eps>0$ small to be chosen.  
In view of the definition of $\underline U$, for $t$ sufficiently large, we have 
\[
\underline U(t, c^* t + \sigma \sqrt{t},\cdot) \leq \eps (\sigma \sqrt{t} + s(t)) e^{-\lambda^*(\sigma\sqrt{t} + s(t))}.
\]
Choosing 
\begin{equation}\label{aug3148}
s(t) = \frac{3}{2\lambda^*}\log(1 + t),
\end{equation}
using (\ref{e:u_lower_bound}), and adjusting $\eps$ as necessary, we see that
\begin{equation}\label{aug3138}
\underline U(t, c^*t + \sigma \sqrt{t}, \cdot)
\leq \frac{C e^{-\sigma \sqrt{t}}}{a_0 t}
\leq u(t, c^*t + \sigma \sqrt t, \cdot).
\end{equation}
for all $t \geq T_0$.  In addition, because of (\ref{aug3128}), 
it is easy to see that translating $\underline U$ further to the left, we may ensure that 
\begin{equation}\label{aug3140}
\underline U(T_0,x,\theta)\leq u(T_0,x,\theta),
\end{equation}
for all $x\le c^*T_0+\sigma\sqrt{T_0}$ and all $\theta\in\Theta$.  
%
%
%
%
The combination of (\ref{aug3134}), (\ref{aug3136}), (\ref{aug3138}) and (\ref{aug3140}) 
the inequalities above, along with the maximum principle, implies that
\begin{equation}\label{aug3142}
\underline U(t,x,\theta) \leq u(t,x,\theta),
\end{equation}
for all $t\ge T_0$, all  $x \leq c^*t + \sigma\sqrt{t}$, and all $\theta\in\Theta$.

To conclude, we need to understand where the level set of height $m$ of $\underline U$ is.  We see from
(\ref{aug3144})
that there exists $L_m$ such that if $x < - L_m$ then 
\[
\underline U(t,c^*t + x -  s(t),\theta) > m.
\]
Thus, (\ref{aug3148}) and (\ref{aug3142}) mean that
\[
\{x \in \R \, : \, \exists \theta \in \Theta, u(t,x,\theta) = m\}
\subset \left[ c^*t - \frac{3}{2\lambda^*} \log(1+t) - L_m, \infty\right).
\]
This finishes the proof of the lower bound in Theorem~\ref{p:kpp}.
  
\begin{figure}
\begin{center}
\includegraphics[width=.9\linewidth]{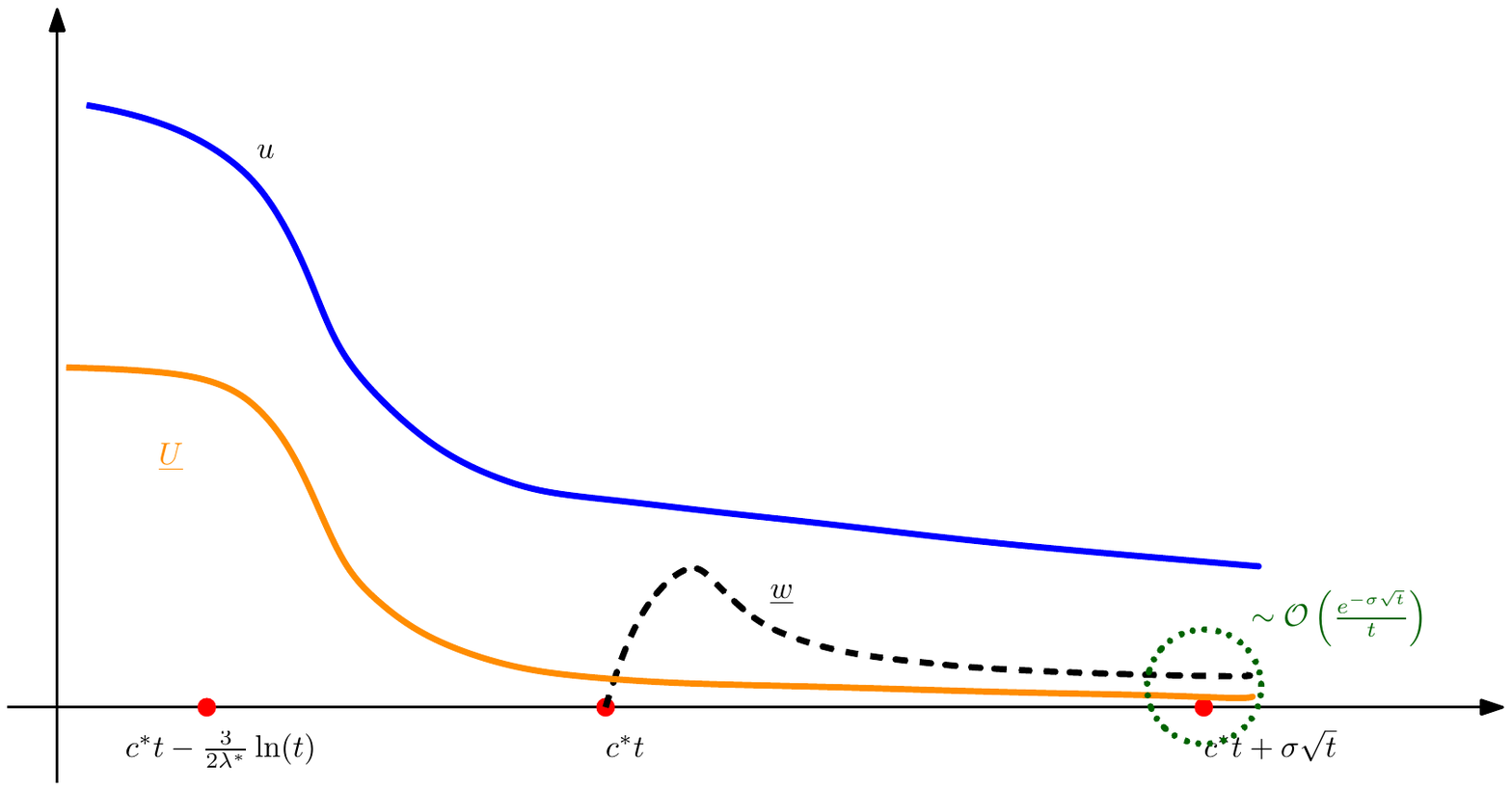}
\end{center}
\caption{A sketch of the sub-solution $\underline U$, the solution $u$ of the Fisher-KPP problem, and of the solution of the
linearized problem with the Dirichlet boundary condition at $x=c^*t$.}
\label{figure2} 
\end{figure}

\section{The proof of Proposition~\ref{p:p_decay}}\label{s:p_decay}

In this section, we prove Proposition~\ref{p:p_decay}.
The proof of the upper bound in (\ref{mar1822}) 
is easier than for the lower bound, and this is what we will do first. 
Essentially, the remainder of the paper will then be devoted to the proof of 
the lower bound in (\ref{mar1822}). 

\subsection{The self-adjoint form}

Our first step is to re-write \eqref{e:p_equation} in a self-adjoint
form. 
Let us set 
\begin{equation}\label{eq:nubis}
\mu = a (Q^*)^2,~~
a=\Big(\frac {1}{|\Theta|}\int_\Theta (Q^*)^2 d\theta\Big)^{-1}.
\end{equation}
Then we have an identity 
\begin{eqnarray}\label{e:p_equationbis}
D p_{xx} + \Delta_\theta p+ \frac{2}{Q^*} \nabla_\theta Q^*\cdot\nabla_\theta p
%
%
= \frac{1}{\mu} \Big[
\left( D\mu p_x\right)_x +\nabla_\theta\cdot( \mu \nabla_\theta p) \Big].
\end{eqnarray}
In order to re-write the spatial drift term in the right side of (\ref{e:p_equation}), we look for a corrector
$\beta$ that satisfies  
\begin{eqnarray}\label{sep102}
&&\Delta_\theta\beta=2\lambda^* D + A -r \quad \hbox{ in $\Theta$},\\
&&\pdr{\beta}{\nu_\theta}=0\hbox{ on $\partial\Theta$},\nonumber
\end{eqnarray}
with some $r\in\Rm$.
%
The solvability condition
for (\ref{sep102}) is 
\begin{eqnarray}\label{aug3170}
r =\int_\Theta[2\lambda^* D(\theta) + A(\theta)]\mu(\theta)\farc{d\theta}{|\Theta|}
=a\int_\Theta[2\lambda^* D(\theta) + A(\theta)](Q^*(\theta))^2
\farc{d\theta}{|\Theta|}=c^*.
\end{eqnarray}
We used (\ref{rel3}) and (\ref{eq:nubis}) in the last step above. Thus, (\ref{e:p_equation}) can be recast as
\begin{equation}\label{aug3156}
(1 - \omega)\mu p_\tau = \mathcal{L}p, 
\end{equation}
with the operator ${\cal L}$ 
\begin{equation}\label{aug3164}
{\cal L}p=  
\left( D \mu p_x\right)_x +
\nabla_\theta\cdot( \mu \nabla_\theta p) 
- \left(\Delta_\theta \beta +{c^*} \right) p_x.
\end{equation}
Note that the  average of the advection term in $x$ in~(\ref{aug3164}) equals to $c^*$.  

We now state a lemma regarding almost-linear solutions to (\ref{aug3156}) and its adjoint.  
The latter will be crucial in the proof of the upper bound for $p$.  The former will be required later. 
We denote by~${\cal L}^*$ the formal adjoint of the
operator ${\cal L}$ with respect to the Lebesgue measure, and set
\begin{equation*}
\mathcal{C}_{\tau} = \left[ c^*\tau, +\infty \right) \times \Theta,
\end{equation*}

\begin{lemma}\label{lem:weightzeta}
There exist functions $\zeta$ and $f$ solving
\begin{equation}\label{sep106}
\left\{
\begin{array}{ll}
\mu\partial_\tau \zeta = \mathcal{L}  \zeta,&~
\hbox{ on $ \mathcal{C}_\tau $,}\medskip\\
\pdr{\zeta}{\nu_\theta}=0,&\hbox{ on $\partial\Theta$}, \medskip\\
\zeta(\tau,c^*\tau,\cdot) = 0,
\end{array}
\right.
	~~\text{ and }~~
		\begin{array}{ll}
	\mu \partial_\tau f = -\mathcal{L}^* f,&~
\hbox{ on $\mathcal{C}_\tau $,}\medskip\\
		\pdr{f}{\nu_\theta}=0&\hbox{ on $\partial\Theta$},\medskip\\
	f(\tau,c^*\tau,\cdot) = 0,
	\end{array}
%
\end{equation}
such that $f_\tau, \zeta_\tau \leq 0$. Moreover,  there exists a constant $C>0$ such that all $x \geq c^* \tau$, 
\begin{equation*}
C^{-1} \left( x - c^* \tau \right) \leq \zeta(t,x,\theta), f(t,x,\theta) \leq C \left( x - c^* \tau \right),
\end{equation*}
and $|\partial_\tau f|, |\partial_\tau \zeta| \leq C$.
\end{lemma} 
We omit the proof as it is very close to~\cite{HNRR12}.

\subsection{The proof of the upper bound}\label{sec:upper_bound}

We now prove the upper bound in (\ref{mar1822}), namely,
there exists a positive constant such that
\begin{equation}\label{sep110}
p(\tau,x,\theta) \leq \frac{C_0(x-c^*\tau)}{(\tau+1)^{3/2}},
\end{equation}
for all $\tau >0$, $x>c^*t$ and $\theta\in\Theta$.  
We use a standard strategy: a Nash-type inequality is used
to obtain the $L^2$ decay in terms of the $L^1$ norm,
and then the uniform decay follows by a duality argument.
%
%
%

We first derive an $L^1-L^2$ bound.  
Using (\ref{aug3156})-(\ref{aug3164}), integrating by parts gives 
that for any~$\tau > 0$, we have
\begin{equation}\label{eq:p_diff_inequality}
\frac{1-\omega}{2} \frac{d}{d\tau}\int_{\cC_\tau} 
\mu(\theta) p(\tau,x,\theta)^2 dxd\theta
= - \int_{\cC_\tau} \mu(\theta) \left[ 
D(\theta) p_x(\tau,x,\theta)^2 + |\nabla_\theta p(\tau,x,\theta)|^2\right] dxd\theta.
\end{equation}
The dissipation in the 
right side may be estimated using a Nash type inequality for half-cylinders
of the form $\Omega = [0,\infty) \times \Theta$, with $\Theta\subset\Rm^d$, 
for functions such that $\phi(0,\cdot) \equiv 0$:
\begin{equation}\label{aug3152}
\Vert \nabla \phi \Vert_2^2
\geq C
\left(1+ \left( \frac{\Vert \phi \Vert_2}{\Vert x\phi \Vert_1} \right)^\frac{10d}
{3(3+d)}\right)^{-1}{ \Vert \phi \Vert_2^2 
\left( \frac{\Vert \phi \Vert_2}{\Vert 
x\phi \Vert_1} \right)^\frac{4}{3} }.
\end{equation}
The proof of the 
one-dimensional version of (\ref{aug3152}) can be found in~\cite{HNRR12}.
We describe the required modifications for $d>1$ in \Cref{sec:nash}.
This gives:
\begin{equation}\label{eq:Nash_application}
\int_{\cC_\tau} \mu(\theta) \left[ D(\theta) p_x(\tau,x,\theta)^2 + 
|\nabla_\theta p(\tau,x,\theta)|^2\right] dxd\theta
\geq 
CI_2(\tau)
\Big(1 + \left(\frac{I_2(\tau)^{1/2}}
{I_1(\tau)}\right)^{\farc{10d}{3(3+d)}}\Big)^{-1}
\left( \frac{I_2(\tau)^{1/2}}{I_1(\tau)}\right)^{\farc{4}3} .
\end{equation}
Here, we have defined
\[
\begin{split}
&I_1(\tau):= \int_{\cC_\tau} 
\mu(\theta) (x-c^*\tau)p(\tau,x,\theta) dxd\theta,  ~~\text{ and},\\
&I_2(\tau):= \int_{\cC_\tau} \mu(\theta)p(\tau,x,\theta)^2 dxd\theta.
\end{split}
\]
We point out that we used in~\eqref{eq:Nash_application} that $\mu$ is bounded uniformly away from $0$ and $\infty$.

Next, we look at
\[
I(\tau):{=} \int_{\cC_\tau} \mu(\theta) f(\tau,x,\theta) p(\tau,x,\theta) dx d\theta,
\]
with $f$ as in (\ref{sep106}). 
If $\omega \equiv 0$, then $I(\tau)$ is a conserved quantity. In general, following the proof of~\cite[Lemma~5.4]{HNRR12}, one can show that
there exists a constant $C>0$ such that
\begin{equation}\label{sep112}
C^{-1} I(0) \leq I(\tau) \leq C\left(\int_{\cC_0} p_0 dxd\theta + I(0)\right).
\end{equation}
Using \Cref{lem:weightzeta}, we see that $I(\tau)$ and $I_1(\tau)$ are comparable:
\[
\frac{1}{C}\int_{\cC_\tau}\mu(\theta)f(\tau,x,\theta) p(\tau,x,\theta) dxd\theta
\leq I_1(\tau)
\leq C\int_{\cC_\tau}  \mu(\theta)f(\tau,x,\theta) p(\tau,x,\theta) dxd\theta.
\]
As a consequence, we have
\begin{equation}\label{eq:conservation}
C^{-2} I_1(0)\leq C^{-1} I_1(\tau)
\leq N :=\left( \int_{\cC_0} p_0 dxd\theta + I_1(0)\right).
\end{equation}

Using~\eqref{eq:conservation} together 
with~\eqref{eq:p_diff_inequality} and~\eqref{eq:Nash_application}, we obtain
\begin{equation}\label{eq:Nash_diff_inequality}
	\left(\frac{I_1(\tau)^{4/3}}{I_2(\tau)^{5/3}} + \frac{I_1(\tau)^{(4-2d)/(3+d)}}{I_2(\tau)^{5/(3+d)}}\right)I_2'(\tau)
		\leq - \frac{1}{C(1 - \omega(\tau))}.
\end{equation}
An elementary argument, starting with this 
differential inequality, using the decay assumptions on~$\omega$ and (\ref{eq:conservation}), gives
an upper bound
\begin{equation}\label{sep120}
I_2(\tau) 
\leq \frac{CN^2}{(\tau + 1)^{3/2}},
\end{equation}
regardless of the cross-section dimension $d\ge 1$. In other words, we have the
bound
\begin{equation}\label{sep306}
\|p(\tau,\cdot)\|_{L^2\left(\mathcal{C}_\tau\right)}\leq \frac{C}{(\tau + 1)^{3/4}}\int_{{\mathcal C}_0}(1+x)|p_0(x,\theta)|dx d\theta.
\end{equation}

%
%
%

We may now apply the standard duality argument. Let $S_\tau$ be the solution operator mapping~$p_0$ to $p(\tau,\cdot)$.
The bound (\ref{sep120}) applies that $S_\tau^*$ satisfies  
\begin{equation}\label{sep304}
|S_\tau^*p_0|\le  \frac{C(1+x-c^*\tau)}{(\tau + 1)^{3/4}}
\|p_0\|_{L^2\left(\mathcal{C}_0\right)}.
\end{equation}
However, $S_\tau^*$ is the solution operator for a parabolic equation of the same type, except for the reverse drift direction, thus it also
obeys the bound (\ref{sep306}), and hence $S_\tau$ itself obeys (\ref{sep304}) as well. 
Decomposing $S_\tau =S_{\tau/2}\circ S_{\tau/2}$ and applying the bounds  (\ref{sep306}) and (\ref{sep304})
separately, we get
\begin{equation}\label{sep308}
|p(\tau,x,\theta)|\le  \frac{C(1+x-c^*\tau)}{(\tau + 1)^{3/2}}
\int_{\cC_0} (1+x)p_0(x,\theta) dxd\theta .
\end{equation}
%
This proves (\ref{sep110}) for $x>c^*\tau+1$. 
However, as $p(\tau,c^*\tau,\cdot)=0$, using the parabolic regularity for~$x\in(c^*\tau,c^*\tau+1)$, we obtain the upper bound  (\ref{sep110}) for
all $x>c^*\tau$.~$\Box$

\subsection{The lower bound for $p$}

We now prove the lower bound on $p$ in Proposition~\ref{p:p_decay}, namely,
there exists a positive constant such that
\begin{equation}\label{sep110bis}
p(\tau,x,\theta) \geq \frac{(x-c^*\tau)}{C_0(\tau+1)^{3/2}},
\end{equation}
for all $\tau >0$, $x>c^*t$ and $\theta\in\Theta$.  

\subsubsection*{Approximate solutions}

For the proof of Proposition~\ref{p:p_decay} will make use of  approximate solutions of our problem that
satisfy the bounds claimed in this Proposition.
Let $Q_\lambda$ be the eigenfunction in (\ref{mar1806}), and set 
\begin{equation}\label{e:chi}
\chi = -\frac{1}{Q_\lambda}\frac{\partial Q_\lambda}{\partial\lambda}\Big|_{\lambda=\lambda^*},
\end{equation}
and
\begin{equation}\label{e:overline_Dbis}
	\overline{D}:{=} \left(\int_{\Theta} \left( D + c^* \chi - 2\lambda^* D \chi - A \chi\right) (Q^*)^2 d\theta\right) \left( \int_{\Theta} (Q^*)^2 d\theta \right)^{-1}.
\end{equation}
To see that $\overline{D} > 0$, we differentiate~\eqref{rel4} in $\lambda$ to obtain
\begin{equation}\label{sep402}
\int_\Theta \left[2Q\pdr{Q}\lambda \left(c'(\lambda)\lambda + c(\lambda) - 2\lambda D - A\right) + Q^2\left( c''(\lambda)\lambda + 2c'(\lambda) - 2 D - A\right) \right]d\theta=0.
\end{equation}
Evaluating (\ref{sep402}) at $\lambda = \lambda^*$, we obtain, as $c'(\lambda^*) = 0$:
\[
	0 = \int_\Theta \left[- 2 (Q^*)^2 \chi \left(c^* - 2\lambda^* D - A\right) + (Q^*)^2\left( c''(\lambda)\lambda - 2 D - A\right) \right]d\theta.
\]
Now, (\ref{e:overline_Dbis}) and  \eqref{rel4} show that this is
\[
	c''(\lambda^*) \lambda^* \int_\Theta (Q^*)^2 d\theta = 2\overline{D} \int_\Theta (Q^*)^2 d\theta.
\]
Since $c''(\lambda^*) > 0$ by \Cref{propspec}, we conclude that $\overline{D} > 0$.

The approximate solutions are described by the following analogue of~\cite[Proposition~5.2]{HNRR12}.
\begin{prop}\label{p:approximate_soln}
Let $\overline \chi \in \R$, then there is a function $S(\tau, x,\theta)$ such that, for any $\sigma>0$,
\begin{equation}\label{e:approximate_soln}
(1 - \omega ) \frac{\partial S}{\partial \tau} - D S_{xx} - 
 \Delta_\theta S + \left( 2 \lambda D + A \right) S_x  - \farc{2}{Q^*} \nabla_\theta {Q^*}\cdot\nabla_\theta S
 = O(\tau^{-3})
\end{equation}
and
\begin{equation}\label{e:approximate_gaussian}
\left|S(\tau, x,\cdot) - \frac{x-c^*\tau + \chi + \overline\chi}{\tau^{3/2}} e^{- \frac{(x-c^*\tau)^2}{4\overline{D} \tau}}
\right|\leq C \tau^{-3/2} \left( \frac{x- c^*\tau}{\sqrt{\tau}}\right)^2 + O(\tau^{-2}),
\end{equation}
for all $x \in [c^*\tau, c^*\tau + \sigma \sqrt{\tau}]$.  The constant $C$ depends on $\sigma$.
\end{prop}
The approximate solutions do approximate true solutions on $ [c^*\tau, c^*\tau + \sigma \sqrt{\tau}]$, as seen from the following.
\begin{prop}\label{p:approximate_soln2}
Fix $\sigma >0$, and let $S$ be as in Proposition~\ref{p:approximate_soln}.  Suppose that $\xi$ satisfies for $\tau > 0$,
\begin{equation}\label{apr102}
\left\{
\begin{array}{ll}
(1 - \omega)\xi_\tau = D \xi_{xx} +  \Delta_\theta\xi
- \left( 2\lambda^* D + A \right) \xi_x + \frac{2 }{Q^*}\nabla Q^* \cdot\nabla\xi,& \qquad x \in [c^*\tau, c^*\tau + \sigma \sqrt{\tau}],\medskip\\
 \xi(\tau, c^*\tau,\cdot) = S(\tau,c^*\tau, \cdot),& \medskip \\ 
\xi(\tau,c^*\tau + \sigma\sqrt{\tau},\cdot) = S(\tau, c^*\tau + \sigma \sqrt{\tau}, \cdot).
\end{array}
\right.
\end{equation}
Then there is a positive constant $\tau_0$ such that, if $\tau \geq \tau_0$ and $x-c^*\tau \in(0,\sigma\sqrt{\tau})$, then
\[
	| \left( \xi - S \right)(\tau, x,\cdot)| \leq \frac{C}{\tau^{3/2}}.
\]
\end{prop}

The proof of \Cref{p:approximate_soln2} is a relatively straightforward energy estimate of the difference~$\xi - S$ that 
can be obtained almost exactly as in~\cite[Proposition~5.3]{HNRR12}.

\subsubsection*{The size of the solution at distance $O(\sqrt{\tau})$}

Another key step is to establish the magnitude of $p$ at 
distances of the order $O(\sqrt{\tau})$ from $x=c^* \tau$.   With the following proposition, we control $p$ at the endpoints of the interval $[c^*\tau, c^* \tau + \sigma \sqrt\tau]$.  Then, the previous propositions allow us to control $p$ in the remainder of the interval as $S$ 
approximates~$p$.
\begin{prop}\label{p:p_in_front}
Let $p$ be as in Proposition~\ref{p:p_decay}.  There are constants $\sigma > 0$ and $C_0 > 0$ so that
\begin{equation}\label{sep416}
\frac{1}{C_0\tau} \leq p(\tau, c^*\tau + \sigma \sqrt{\tau}) \leq \frac{C_0}{\tau}
\end{equation}
whenever $\tau \geq 1$.
\end{prop}

\subsubsection*{Sketch of the proof of Proposition~\ref{p:p_decay}}

We now outline how to combine \Cref{p:approximate_soln,p:approximate_soln2,p:p_in_front}  to obtain 
 the lower bound  in \Cref{p:p_decay}. 
%
%
\Cref{p:p_in_front} controls $p$ at the point $c^*\tau + \sigma \sqrt\tau$ in a way consistent with~\eqref{mar1822}.  
On the other hand, by choosing $\overline \chi = -(1 + \|\chi\|_\infty)$ in \Cref{p:approximate_soln}, the combination of 
\Cref{p:approximate_soln,p:approximate_soln2} allows us to build a sub-solution $\xi^-$ to $p$.  Then, re-applying \Cref{p:approximate_soln2}, 
we see that $\xi^-$ satisfies the bounds in~\eqref{mar1822} except on a finite interval $[c^*\tau, c^*\tau + x_0]$, for some $x_0$.  By the 
comparison principle, we may then transfer these bounds to $p$ and use parabolic regularity to remove the condition on~$x_0$, finishing the proof 
of the claim. Thus, it remains to prove Propositions~\ref{p:approximate_soln} and \ref{p:p_in_front},
which is done in the rest of this paper.

\subsection{The proof of Proposition \ref{p:approximate_soln}}

Our strategy is the same as in \cite[Proposition~5.2]{HNRR12}, though the details are different, so we include a sketch of the proof 
for reader's convenience.   We  begin with the multi-scale expansion
\[
S(\tau, x, \theta)
= \frac{1}{\tau} \left(S^0(z) + \frac{S^1(z,\theta)}{\sqrt{\tau}} + \frac{S^2(z,\theta)}{\tau} + \frac{S^3(z,\theta)}{\tau^{3/2}}\right),
~~z = \farc{x - c^*\tau}{\sqrt{\tau}}.
\]
Plugging this into the left hand side of \eqref{e:approximate_soln}, we obtain the equation
\begin{equation}\label{apr110}
\begin{split}
&\frac{(1-\omega)}{\tau} \left[ -\frac{S^0}{\tau} - \frac{3S^1}{2\tau^{3/2}} - \frac{2S^2}{\tau^2} - \frac{5S^3}{2\tau^{5/2}}\right]
	+ \frac{(1-\omega)}{\tau} \left[ -c^* \frac{S^0_z}{\tau^{1/2}} - c^* \frac{S^1_z}{\tau} - c^*\frac{S^2_z}{\tau^{3/2}} - c^* \frac{S^3_z}{\tau^2}\right]\\
	& + \frac{(1-\omega)}{\tau} \left[ -\frac{z}{2} \frac{S^0_z}{\tau} - \frac{z}{2} \frac{S^1_z}{\tau^{3/2}} - \frac{z}{2} \frac{S^2_z}{\tau^2} - \frac{z}{2} \frac{S^3_z}{\tau^{5/2}}\right]
		+ \frac{D}{\tau} \left[ -\frac{S^0_{zz}}{\tau}  - \frac{S^1_{zz}}{\tau^{3/2}} - \frac{S^2_{zz}}{\tau^2} - \frac{S^3_{zz}}{\tau^{5/2}} \right]\\
	& + \frac{1}{\tau} \left[\frac{LS^1}{\tau^{1/2}} + \frac{LS^2}{\tau} + \frac{LS^3}{\tau^{3/2}}\right]
	 + \frac{\left( 2\lambda^* D + A \right)}{\tau} \left[ \frac{S^0_z}{\sqrt\tau} + \frac{S^1_z}{\tau} + \frac{S^2_z}{\tau^{3/2}} + \frac{S^3_z}{\tau^2}\right]
	 = 0.
\end{split}
\end{equation}
Here, we have defined the operator 
\[
L = \Delta_{\theta} + \frac{2  }{Q^*}\nabla_\theta Q^*\cdot\nabla_\theta.
\]
Grouping the terms of order $\tau^{-3/2}$ in (\ref{apr110}), we obtain
\begin{equation}\label{apr106}
L S^1 =  \left( c^* - 2\lambda^* D - A\right) S^0_z.
\end{equation}
It is easy to verify that (\ref{apr106}) has a solution of the form
\begin{equation}\label{e:S_1}
	S^1 = \chi_0 S^0_z + \phi_1,~~~\chi_0=\chi+\overline\chi,
\end{equation}
where $\phi_1$ only depends on $z$. The terms of order $\tau^{-2}$ in (\ref{apr110}) give
\begin{equation}\label{sep404}
-S^0 - c^* S^1_z - \frac{z}{2}S^0_z - D S^0_{zz} + L S^2 + 2\lambda D S^1_z + A S^1_z = 0.
\end{equation}
Using expression (\ref{e:S_1}) for $S^1$,  multiplying (\ref{sep404}) by $(Q^*)^2$ and integrating in $\theta$, we obtain
\begin{equation}\label{sep406bis}
S^0 + \frac{z}{2} S^0_z + \overline{D} S^0_{zz} = 0
\end{equation}
with $\overline D$ as in (\ref{e:overline_Dbis}), so that
\begin{equation}\label{e:S_0}
	S^0(z) = z \exp\left\{ - \frac{z^2}{4\overline{D}}\right\}.
\end{equation}
%
With this in hand, we return to \eqref{sep404} that we write   as
\begin{equation}\label{sep410}
LS^2 = (c^* - 2\lambda D + A)(S^1)_z + (D + c^* \chi_0 - 2\lambda^* D\chi_0 + A \chi_0 - \overline{D}) S^0_{zz}.
\end{equation}
One solution of (\ref{sep410}) is
\[
	S^2(z,\theta) = \chi_0(\theta) (\phi_1)_z(z) + \hat{S}^2(\theta) S^0_{zz}(z), 
\]
where $\hat{S}^2(\theta)$ is any solution to
\begin{equation}\label{e:S_2}
	L \hat{S}^2 = D + c^*\chi_0 - 2\lambda^* D \chi_0 + A \chi_0 - \overline{D},
\end{equation}
with the Neumann boundary conditions. The definition of $\overline{D}$ ensures that solution of (\ref{e:S_2}) exists.

Continuing, we examine the terms of order $\tau^{-5/2}$ to obtain  
\begin{equation}\label{sep406}
-\frac{3}{2}S^1 + \overline\omega c^* S^0_z - c^* S^2_z - \frac{z}{2} S^1_z - D S^1_{zz} + 2\lambda D S^2_z + L S^3 + A S^2_z = 0.
\end{equation}
Here, we replaced $\omega$ by $\overline\omega / \tau$ at the expense of lower order terms which we may absorb 
into the~$O(\tau^{-3})$ term in \eqref{e:approximate_soln}.  
Multiplying by $(Q^*)^2$ and integrating over $\theta$ yields the solvability condition for $S^3$:
\begin{equation}\label{sep408}
-\frac{3}{2} \phi_1 - \frac{z}{2} \left(\phi_1\right)_z - \overline{D} (\phi_1)_{zz}
= \left(3\beta_1 - \overline\omega c^*\right) S^0_z + z \beta_1 S^0_{zz} + \beta_2 S^0_{zzz}.
\end{equation}
Here, we have defined
\[ 
\beta_1:{=} \frac{1}{2}\frac{\int_\Theta \chi_0 (Q^*)^2 d\theta}{\int_\Theta (Q^*)^2 d\theta}, ~~~ 
\beta_2 :{=} \frac{\int_\Theta (c^*\hat{S}^2 - 2\lambda D\hat{S}^2 - A\hat{S}^2  + D) (Q^*)^2 d\theta}{\int_\Theta (Q^*)^2 d\theta}.
\]
We may now choose $\phi_1$ to be the unique solution to (\ref{sep408})
with $\phi_1(0) = 0$ and $(\phi_1)_z(0) = 0$.  Since~$S^0$ and its derivatives are bounded, there exists a 
constant $C$ such that $|\phi_1(z)| \leq C z^2$ for all~$|z| \leq \sigma$.  For the sake of clarity, we write $\phi_1 = z^2 \overline\phi$,
with a bounded function $\overline\phi$.  

Finally, grouping the $\tau^{-3}$ terms together and setting them to zero, we get an equation for $S^3$.  It follows   
from  the elliptic regularity 
theory that, for $z \leq \sigma$, $S^3$ is uniformly bounded.  To summarize, we have found an approximate solution,  in the
sense that (\ref{e:approximate_soln}) holds, of the form
\begin{equation}\label{sep414}
S
= \frac{x - c^* \tau}{\tau^{3/2}} e^{-\frac{(x-c^*\tau)^2}{4\overline{D}\tau}} + \chi_0\frac{1 - \frac{z^2}{4\overline{D}\tau}}{\tau^{3/2}} e^{-\frac{(x-c^*\tau)^2}{4\overline{D}\tau}} + \frac{x-c^*\tau}{\tau^{5/2}}\overline\phi(z) + \frac{S^2}{\tau^2} + \frac{S^3}{\tau^{5/2}}.
\end{equation}
It also clearly satisfies the condition~\eqref{e:approximate_gaussian}.  
This concludes the proof.~$\Box$

\subsection{Understanding $p$ at $x-c^*\tau \sim O(\sqrt\tau)$: the proof of \Cref{p:p_in_front}}\label{s:p_in_front}

%
%
The lower bound in (\ref{sep416}) is a consequence of an integral bound. 
\begin{lemma}\label{p:p_in_front_weak}
There exists a time $T_0>0$ and constants $c_0$, $B$, and $N$, 
depending only on the initial data, such that for any $\tau > T_0$ 
there exists a 
set $I_\tau\subset[c^*\tau + N^{-1}\sqrt{\tau}, c^*t + N \sqrt{\tau}]$ 
with $|I_\tau| \geq B \sqrt{\tau}$ and with
\begin{equation}\label{sep422}
\frac{1}{c_0\tau}
\leq \int_{\Theta} p(\tau,x,\eta) d\eta.
\end{equation}
\end{lemma}

Proposition \ref{p:p_in_front} follows from Lemma~\ref{p:p_in_front_weak} and a standard heat kernel bound. 
Indeed, let us assume that $\omega = 0$, as we may otherwise apply the time change
\[
d\tau' = \frac{d\tau}{1 - \omega(\tau)}.
\]
Let $\Gamma$ be the heat kernel for  (\ref{aug3156})-(\ref{aug3164})  
with the Dirichlet boundary condition at $x=c^*t$. That is, the solution of
\begin{eqnarray}\label{eq:self_adjoint_heat}
&&\mu \psi_\tau  = \left( D \mu \psi_x\right)_x +
\nabla_\theta\cdot( \mu \nabla_\theta \psi) 
- \left(\Delta_\theta \beta  +{c^*} \right) \psi_x,~~\tau>s,~x>c^*\tau,~\theta\in\Theta,\nonumber\\
&&\pdr{\psi}{\nu_\theta}=0,\hbox{ on $\partial\Theta$},\\
&&\psi(\tau,c^*\tau,\cdot)=0,\nonumber\\
&&\psi(s,\cdot)=\bar\psi,\nonumber
\end{eqnarray}
 can be written as
\begin{equation}\label{sep426}
\psi(\tau,x,\theta) = \int_{{\cal C}_s} \Gamma(\tau,x,\theta,s,y,\eta) \bar\psi(y,\eta) \mu(\eta) dyd\eta.
\end{equation}
%
As in~\cite{HNRR12}, one can show the following, starting with the standard heat kernel bound
in a cylinder. Set $\Phi(s)=s$ for $s\in[0,1]$ and $\Phi(s)=\sqrt{s}$ for $s>1$, then for 
all $\delta > 0$, there exists a constant $K$ such that
\begin{equation}\label{sep420}
\Gamma(\tau,x,\theta,s,y,\eta)
\geq \frac{1}{K \Phi(\tau-s)} \exp \left\{ - K \frac{|x - y|^2 + |\theta - \eta|^2}{\Phi(\tau - s)}\right\}
\end{equation}
whenever $R>0$, $\tau\in (s, s+R^2]$, and $x,y\in (c^*\tau + \xi - \delta R, c^*\tau + \xi + \delta R)$. A straightforward computation using
(\ref{sep426}) going from the time $s=\tau/2$ to $\tau$
shows that the integral bound (\ref{sep422}),
combined with the pointwise lower bound~(\ref{sep420}) on the heat kernel, lead to a pointwise lower bound on $p$ in 
Proposition~\ref{p:p_in_front}.~$\Box$

\subsection{Proof of \Cref{p:p_in_front_weak} }

\subsubsection*{An exponentially weighted estimate}

As in~\cite{HNRR12}, one may show that 
for all $\alpha > 0$, there exists a function $\eta_\alpha$ that satisfies 
\begin{eqnarray}\label{sep430}
&&\mu \partial_\tau \eta_\alpha = - \mathcal{L}^* \left( \mu \eta_\alpha \right) 
+ \aleph(\alpha) \mu \eta_\alpha,~
\hbox{ on $\mathcal{C}_\tau $,} 
\\
&&\pdr{\eta_\alpha}{\nu_\theta}=0\hbox{ on $\partial\Theta$,}
\nonumber \\
&&	\eta_\alpha(\tau,c^*\tau,\cdot) = 0.\nonumber
\end{eqnarray}
as well as the exponential bounds
\begin{equation}\label{eq:exponential}
 \frac{e^{\alpha (x-c^*\tau)} - e^{-\alpha (x-c^*\tau)}}{C\alpha}
		\leq \eta_\alpha(t,x,\theta)
		\leq C \frac{e^{\alpha (x-c^*\tau)} - e^{-\alpha (x-c^*\tau)}}{\alpha}.
\end{equation}
The  eigenvalue $\aleph(\alpha)$ in (\ref{sep430}) behaves as
\begin{equation}\label{sep432}
\aleph(\alpha) = \aleph_0 \alpha^2 + O(\alpha^3),
\end{equation}
as $\alpha$ tends to zero, with some $\aleph_0>0$.  
Moreover, we have  
\begin{equation}\label{sep434}
\begin{split}
	&|\partial_\tau \eta_\alpha| \leq C
		~~\qquad \text{ for all } x \in [c^*\tau, c^*\tau + \alpha^{-1}], ~\text{ and }\\
	&|\partial_\tau \eta_\alpha| \leq C \alpha \eta_\alpha
		~~~\text{ for all } x \geq c^*\tau + \alpha^{-1}.
\end{split}
\end{equation}
With this in hand, we define  
\begin{equation}\label{eq:integral_quantitiesbis}
V_\alpha(\tau) = \left( 1 - \omega(\tau) \right) \int_{\mathcal{C}_\tau} \mu(\theta) \, \eta_{2\alpha}(\tau,x,\theta) \, p(\tau,x,\theta)
 \, q(\tau,x,\theta) \, dxd\theta. 
\end{equation}
Here, we write
\begin{equation}\label{sep2604}
p = q \zeta,
\end{equation}
and $\zeta$ is as in~\Cref{lem:weightzeta}.
\Cref{p:p_in_front_weak} is a consequence of the following estimate. 
\begin{lem}\label{lem:V_decay}
There is a constant $C_0$ depending on $p_0$ such that
\begin{equation}\label{sep436}
V_{\tau^{-1/2}}(\tau) \leq C_0 \tau^{-3/2}.
\end{equation}
\end{lem}
We first show how to conclude the proof of \Cref{p:p_in_front_weak}  from \Cref{lem:V_decay}.  Note that \eqref{sep436} implies
\begin{equation}\label{sep438}
\left( \int_0^\infty\int_\Theta \frac{e^{2x/\sqrt{\tau}} - e^{-2x/\sqrt{\tau}}}{x} p(\tau,c^*\tau + x, \theta)^2 dxd\theta\right)^{1/2}
\leq \frac{C_0}{\tau},
\end{equation}
Fix $N>0$ to be determined later, then (\ref{sep438}) gives, in particular:
\begin{eqnarray}\label{sep40}
&&\int_{N\sqrt{\tau}} ^\infty\int_\Theta x p(\tau,c^*\tau + x, \theta) dxd\theta
= \int_{N\sqrt{\tau}}^\infty \int_\Theta\frac{e^{x/\sqrt{\tau}}}{\sqrt x} p(\tau,c^*\tau + x, \theta) e^{-x/\sqrt{\tau}} x^{3/2} dxd\theta
\nonumber\\
&&\leq \left(\int_{N\sqrt{\tau}}^\infty\int_\Theta \frac{e^{2x/\sqrt{\tau}}}{x} p(\tau,c^*\tau + x, \theta)^2 dxd\theta\right)^{1/2}
\left(\int_{N\sqrt{\tau}}^\infty\int_\Theta e^{-2x/\sqrt{\tau}} x^3dxd\theta\right)^{1/2}\nonumber\\
&&\leq \frac{C}{\tau} \left(\int_{N\sqrt{t}}^\infty\int_\Theta e^{-2x/\sqrt{\tau}} x^3dxd\theta\right)^{1/2}
\leq C_0 N^3 e^{-N/2}.
\end{eqnarray}
On the other hand, we also have 
\[
\int_0^{\sqrt{\tau}/N}\int_\Theta x p(\tau,c^*\tau + x, \theta) dxd\theta
\leq C_0 N^{-3}.
\]
Hence, choosing $N$ sufficiently large, depending only on the initial data of $p$ and not on time, we have 
\begin{equation}\label{sep602}
\int_{\sqrt{\tau}/N}^{N\sqrt{\tau}}\int_\Theta x p(\tau,c^*\tau + x, \theta) dxd\theta \geq C_0.
\end{equation}
Let us set
\[
I_\tau:{=} \left\{x \in [c^*\tau + \sqrt\tau/N,c^*\tau + N\sqrt{\tau}] : \int_\Theta p(\tau,x, \eta)d\eta \geq \frac{C_0}{4N^2\tau}\right\}.
\]
Then, (\ref{sep602}) implies
\[
\frac{3C_0}{4}
\leq \int_{I_\tau} \int_\Theta (x-c^*\tau) p(\tau,x,\theta)dxd\theta
\leq \int_{I_\tau} \int_\Theta \frac{C_0(x-c^*\tau)^2}{\tau^{3/2}}dxd\theta
\leq |I_\tau|\frac{C_0N^2}{\tau^{1/2}},
\]
and the proof of \Cref{p:p_in_front_weak} is complete.~$\Box$

\subsection{The proof of \Cref{lem:V_decay}}


Throughout this section we use the assumption that $\tau \leq \alpha^{-2}$.  
The proof relies on two observations. First, we have the following energy-dissipation inequality for $V_\alpha$:
\begin{equation}\label{eq:V_alpha}
V_\alpha'(\tau) \leq \ds \left(  \aleph(2\alpha) - \frac{\omega'(\tau)}{1- \omega(\tau)} + C\alpha \omega(\tau) \right)  V_\alpha (\tau)
- 2 D_\alpha(\tau) + \frac{C_0}{(\tau+1)^{5/2}},
\end{equation}
with the dissipation
\begin{equation}\label{sep2602}
D_\alpha=\int_{\mathcal{C}_\tau} \mu(\theta) \eta_{2\alpha}(\tau,x,\theta) \zeta(\tau,x,\theta) \left( D(\theta) \vert q_x(\tau,x,\theta) \vert^2 
+ \vert q_\theta(\tau,x,\theta) \vert^2 \right) dxd\theta.
\end{equation}
Recall that the function $\zeta$ is defined in Lemma~\ref{lem:weightzeta},  and $q$ is as in (\ref{sep2604}). 
Since this computation is quite involved, we delay it for the moment.

The second observation is that the dissipation $D_\alpha$ may be related  to $V_\alpha$ by the inequality
\begin{equation}\label{eq:D_alpha}
	D_\alpha
		\geq \frac{1}{C_0} V_\alpha^{5/3}
\end{equation}
where $C_0$ is a constant depending only on $p_0$ and $\tau \in [0,T]$.  We also delay the proof of~\eqref{eq:D_alpha}.

The combination of~\eqref{eq:V_alpha} and~\eqref{eq:D_alpha} yields the differential inequality
\begin{equation}\label{sep2610}
V_\alpha'
\leq 
\left(\aleph(2\alpha) - \frac{\omega'}{1-\omega} + \frac{C\alpha}{\tau + 1}\right) V_\alpha - \frac{1}{C_0}V_\alpha^{5/3} + \frac{C_0}{(\tau+1)^{5/2}}.
\end{equation}
Let us define
\[
Z(\tau) = (\tau+1)^{3/2} V_\alpha \exp(-\Phi(\tau)),~~\Phi(\tau)=\aleph(2\alpha)\tau +\log(1-\omega(\tau)) + C\alpha \log(\tau+1) .
\]
Note that,  as $\tau\le\alpha^{-2}$, we know, due to the asymptotics  (\ref{sep432}) for $\aleph(2\alpha)$,  that 
\begin{equation}\label{sep2606}
\hbox{$|\Phi(s)|\le C$ for all~$0\le s \le\tau$, }
\end{equation}
with a constant $C>0$ that is independent of~$\alpha>0$ sufficiently small. 
Thus, (\ref{sep436}) would follow if we show that
that $Z$ is uniformly bounded above. However, it follows from (\ref{sep2610}) and (\ref{sep2606}) that~$Z$ satisfies
\[
Z' 
\leq C \frac{Z}{\tau+1} + \frac{C_0}{\tau +1} - \frac{1}{C_0 (\tau+1)} Z^{5/3}.
\]
This implies  
\[
	Z^{5/3} \leq \max\left\{
		C \left(Z + C_0\right), Z(0)^{5/3}\right\}.
\]
Hence, $Z$ is bounded uniformly above.
Thus, to finish the proof of \Cref{lem:V_decay}, it only remains to show~\eqref{eq:V_alpha} and~\eqref{eq:D_alpha}.

\subsubsection*{Proof of the differential inequality \eqref{eq:V_alpha} for $V_\alpha$ }

Differentiating $V_\alpha$, we obtain
\begin{equation}\label{sep2702}
V_\alpha' = \ds - \frac{\omega'}{1- \omega} V_\alpha + 
\left( 1 - \omega \right) \int_{\mathcal{C}_\tau} 
\mu  \left[ \left( \partial_\tau \eta_{2\alpha} \right) p q + \eta_{2\alpha} p_\tau q + \eta_{2\alpha} p q_\tau  \right] dxd\theta.
\end{equation}
Let us re-write the integral in (\ref{sep2702}). 
By the definition of $\eta_{2\alpha}$, we have
\begin{equation*}
\begin{split}
\left( 1 - \omega \right) \int_{\mathcal{C}_\tau} \mu \left( \partial_\tau \eta_{2\alpha} \right) p q \, dxd\theta
&=\left( 1 - \omega  \right) \int_{\mathcal{C}_\tau} \left[ - \mathcal{L}^* \left( \mu \eta_{2\alpha} \right) 
+ \aleph(2\alpha) \mu \eta_\alpha \right] p q \, dxd\theta,\\
&= \ds - \left( 1 - \omega \right) \int_{\mathcal{C}_\tau} \mu \eta_{2\alpha} \mathcal{L}(p q) \, dxd\theta + 
\aleph(2\alpha) V_\alpha.
\end{split}
\end{equation*}
Using equation~\eqref{e:p_equation} for $p$, we deduce 
\begin{equation}\label{sep2710}
V_\alpha' = \ds \left\lbrace \aleph(2\alpha) - \frac{\omega'}{1- \omega} \right\rbrace V_\alpha 
+ \int_{\mathcal{C}_\tau} \mu \eta_{2\alpha} \left[ -  \left( 1 - \omega(\tau) \right) \mathcal{L}(p q) 
+ \mathcal{L}(p) q + (1- \omega) pq_\tau  \right] dxd\theta.
\end{equation}
The last integral requires a bit of work. 
Note that
\[
 \mathcal{L}( p q) = p \mathcal{L}(q) + q \mathcal{L}(p) + 2 D  p_x q_x + 2 p_\theta q_\theta,
 \]
and
\begin{eqnarray*}	
&&(1-\omega) q_\tau = \ds \mathcal{L}(q) + \frac{q}{\zeta} \left( \mathcal{L}\left( \zeta \right) - \left( 1 - \omega \right) \zeta_\tau \right) + 2 D  \frac{\zeta_x}{\zeta} q_x + 2 \frac{\zeta_\theta}{\zeta} q_\theta, \\
&&~~~~~~~~~~~~~~~~
= \ds \mathcal{L}(q) + \frac{q}{\zeta} \left( \mathcal{L}\left( \zeta \right) - \zeta_\tau \right) + \omega   \frac{\zeta_\tau}{\zeta}q+ 2 D \frac{\zeta_x}{\zeta} q_x + 2 \frac{\zeta_\theta}{\zeta} q_\theta.
\end{eqnarray*}
Thus, we may re-write (\ref{sep2710}) as
\[
\begin{split}
V_\alpha'  &= \ds \left(  \aleph(2\alpha) - \frac{\omega' }{1- \omega} \right) V_\alpha 
+ \omega  \ds \int_{\mathcal{C}_\tau} \mu \eta_{2\alpha} \mathcal{L}(p q)  dxd\theta 
+ \omega  \int_{\mathcal{C}_\tau} \mu \eta_{2\alpha} p q \frac{\zeta_\tau}{\zeta} dxd\theta\\
&\qquad - 2 \ds \int_{\mathcal{C}_\tau} \mu \eta_{2\alpha} ( D p_x q_x + p_\theta q_\theta ) dxd\theta  
+ 2 \int_{\mathcal{C}_\tau} \mu \eta_{2\alpha} p \left( D \frac{\zeta_x}{\zeta} q_x + \frac{\zeta_\theta}{\zeta} q_\theta \right) dxd\theta. 
\end{split}
\]
The last two terms in the right side can be combined as
\[
\begin{split}
&- 2 \ds \int_{\mathcal{C}_\tau} \mu \eta_{2\alpha} ( D  p_x q_x + \alpha p_\theta q_\theta ) \, dxd\theta + 
2 \ds \int_{\mathcal{C}_\tau} \mu \eta_{2\alpha} \left( D  \frac{ \zeta_x}{\zeta} q_x 
+   \frac{\zeta_\theta}{\zeta} q_\theta \right) p \, dxd\theta\\
&\quad = - 2 \ds \int_{\mathcal{C}_\tau} \mu \eta_{2\alpha} \left( D  
\left( p_x - \frac{ \zeta_x}{\zeta} p \right) q_x + \left( p_\theta - \frac{ \zeta_\theta}{\zeta} p \right) q_\theta \right) dxd\theta\\
&\quad= - 2 \ds \int_{\mathcal{C}_\tau} \mu \eta_{2\alpha} \zeta \left( D  \vert q_x \vert^2 + \vert q_\theta \vert^2 \right) dxd\theta
		 = - 2 D_\alpha,
\end{split}
\]
hence
\begin{equation}\label{sep902}
V_\alpha' = \ds \left(  \aleph(2\alpha) - \frac{\omega' }{1- \omega } \right) V_\alpha  - 
2 D_\alpha  + \omega  \ds \int_{\mathcal{C}_\tau} \mu \eta_{2\alpha} \left( \mathcal{L}(p q) + p \frac{\zeta_\tau}{\zeta} q \right) dxd\theta .
\end{equation}
The estimate (\ref{eq:V_alpha}) will be complete after estimating the last term in the right side.  We  write
\begin{equation}\label{sep2501}
\int_{\mathcal{C}_\tau} \mu \eta_{2\alpha} \left( \mathcal{L}(p q) + p \frac{\zeta_\tau}{\zeta} q \right) dxd\theta
= \int_{\mathcal{C}_\tau} \left( p q \mathcal{L}^*(\mu \eta_{2\alpha} ) + \mu \eta_{2\alpha} \frac{\zeta_\tau}{\zeta}  p q \right) dxd\theta
\leq \int_{\mathcal{C}_\tau} \mu [\partial_\tau \eta_{2\alpha} ]p q \, dxd\theta,
\end{equation}
as $\zeta_\tau \leq 0$. We  use~\eqref{sep434} to obtain
\begin{equation*}
\begin{split}
\left\vert \int_{\mathcal{C}_\tau} \mu \partial_\tau \eta_{2\alpha} p q  dxd\theta \right\vert
&\leq  \left\vert \int_{c^*\tau}^{c^*\tau + \alpha^{-1}}\int_\Theta \mu \partial_\tau \eta_{2\alpha} p q  dxd\theta \right\vert
			+ \left\vert \int_{c^*\tau + \alpha^{-1}}^\infty \int_\Theta \mu \partial_\tau \eta_{2\alpha} p q  dxd\theta \right\vert \\
		&\leq \int_{c^*\tau}^{c^*\tau + \alpha^{-1}}\int_\Theta \mu p q  dxd\theta
			+ C \alpha \int_{c^*\tau + \alpha^{-1}}^\infty\int_\Theta \mu \eta_{2\alpha} p q  dxd\theta.
\end{split}
\end{equation*}
The second term above is $C\alpha V_\alpha$, as desired.  
For the first term, we  apply the upper bound~\eqref{sep110} for~$p$  and the asymptotics for $\zeta$ in~\Cref{lem:weightzeta} to obtain
\[
\int_{c^*\tau}^{c^*\tau + \alpha^{-1}}\int_\Theta \mu p q  dxd\theta
		\leq \frac{C_0}{(\tau+1)^{3/2}}\int_{c^*\tau}^{c^*\tau + \alpha^{-1}}\int_\Theta p  dxd\theta.
\]
Integrating~\eqref{aug3156}, we see that $\int_{\mathcal{C}_\tau} \mu p dx d\theta$ is non-increasing in time.
Hence, we obtain that
\[
\omega\int_{\mathcal{C}_\tau} \mu \eta_{2\alpha}\left(\mathcal{L}(pq) + p\frac{\zeta_\tau}{\zeta} q\right) dxd\theta
\leq \frac{C_0}{(\tau+1)^{5/2}} + C\frac{\alpha}{\tau + 1} V_\alpha.
\]
Returning to~\eqref{sep902}, we obtain the desired differential inequality
\[
	V_\alpha'
		= \left(  \aleph(2\alpha) - \frac{\omega'(\tau)}{1- \omega(\tau)} + C\frac{\alpha}{\tau+1}\right) V_\alpha  - 2 D_\alpha  + \frac{C_0}{(\tau+1)^{5/2}} .
\]


\subsubsection*{Proof of the 
inequality \eqref{eq:D_alpha} relating $V_\alpha$ and $D_\alpha$~ }

It is helpful to define
\[
\varphi(\tau, z,\theta) = e^{\alpha z_3} q(\tau, c^*\tau + |z|, \theta),
\]
with $(z_1,z_2,z_3) = z \in \R^3$,
and consider the following quantities
\begin{equation}\label{eq:nash_integrals}
\begin{split}
&\hat I_\alpha := \frac{1}{2\pi} \int_{\R^3\times \Theta} 
\varphi(\tau,z,\theta)dzd\theta
= \int_{\cC_\tau} \left( \frac{e^{\alpha (x-c^*\tau)} - e^{-\alpha (x-c^*\tau)}}{\alpha}\right) (x-c^*\tau) q(\tau, x,\theta) \mu \, dzd\theta,\\
&\hat V_\alpha := \frac{1}{2\pi} \int_{\R^3\times \Theta} 
\varphi(\tau,z,\theta)^2 dzd\theta
= \int_{\cC_\tau} \left( \frac{e^{2\alpha (x-c^*\tau)} 
- e^{-2\alpha (x-c^*\tau)}}{\alpha}\right) (x-c^*\tau) 
q^2(\tau,x,\theta) \mu \, dzd\theta,\\
&\hat D_\alpha := \frac{1}{2\pi} \int_{\R^3\times \Theta} 
|\nabla\varphi(\tau,z,\theta)|^2 dzd\theta\\
&\qquad=\int_{\cC_\tau} \left( \frac{e^{2\alpha (x-c^*\tau)} - 
e^{-2\alpha (x-c^*\tau)}}{2\alpha}\right) (x-c^*\tau) 
(\vert \nabla q (\tau,x,\theta)\vert^2 - \alpha^2 q(\tau,x,\theta)^2) 
\mu \, dzd\theta.
\end{split}
\end{equation}
They can be related by the following Nash-type inequality.
\begin{prop}\label{p:Nash}
Let $\Theta\subset \R^d$ be  a smooth, bounded domain,
and   $\Omega = \R^k \times \Theta$.
There exists a constant $C$, depending only on $d$, $k$, and $|\Theta|$ such that if $\phi$ is any function in $L^1(\Omega)\cap H^1(\Omega)$ satisfying Neumann boundary conditions on the boundary $\partial\Theta$, then
\begin{equation}\label{aug3150}
\Vert \nabla \phi \Vert_2^2
\geq C\Big(1+ \left( \frac{\Vert \phi \Vert_2}{\Vert \phi \Vert_1} \right)^\frac{2d(k+2)}{k(k+d)}\Big)^{-1}
{\Vert \phi \Vert_2^2 \left( \frac{\Vert \phi \Vert_2}{\Vert \phi \Vert_1} \right)^\frac{4}{k} }.
\end{equation}
\end{prop}

Inequality (\ref{aug3150})
is a multi-dimensional version of a Nash-type inequality in~\cite{FanKisRyz},
while the one-dimensional version of (\ref{aug3152}) is in~\cite{HNRR12}.
Its proof is in Section~\ref{sec:nash}.

We may apply \Cref{p:Nash} to $\phi$ in the  
cylinder $\R^3 \times \Theta$ to obtain
\begin{equation}\label{eq:D_alpha_hat}
	\hat D_\alpha
		\geq \frac{1}{C}\frac{\hat V_\alpha^{5/3} \hat I_\alpha^{-4/3}}{1 + \hat V_\alpha^{5d/(3d + 9)}\hat I_{\alpha}^{-10d/(3d + 9)}}
\end{equation}
Using the bounds for $\zeta$ in \Cref{lem:weightzeta} 
and the exponential bounds \eqref{eq:exponential}
for $\eta_\alpha$, we see that
\begin{equation}\label{eq:integral_relations}
	C^{-1} \hat V_\alpha \leq V_\alpha \leq C \hat V_\alpha~~~
	\text{ and } ~~~\hat D_\alpha \leq CD_\alpha.
\end{equation}
We claim that
\begin{equation}\label{eq:strong_I_alpha}
	\frac{1}{C_0}
		\leq \hat I_\alpha  \leq C_0
	~~~\text{ and }~~~
	V_\alpha  \leq C_0,
\end{equation}
so that (\ref{eq:D_alpha_hat}) implies
\[
D_\alpha \geq 
\frac{1}{C_0}V_\alpha^{5/3},
\]
which is (\ref{eq:D_alpha}).

To finish, we need to show that~\eqref{eq:strong_I_alpha} holds.  
We begin with the inequality for $\hat I_\alpha$ in (\ref{eq:strong_I_alpha}). 
Let us introduce
\begin{equation}\label{sep2802}
I_\alpha  = (1 - \omega  ) \int_{\mathcal{C}_\tau} \mu(\theta) \eta_\alpha(\tau,x,\theta) p(\tau,x,\theta) dxd\theta.
\end{equation}
We note that 
\[
C^{-1} I_\alpha \leq \hat I_\alpha \leq C I_\alpha,
\]
by~\eqref{eq:exponential}.  Hence, we need only show that $I_\alpha$ is bounded 
away from infinity and zero uniformly in~$\tau$ and $\alpha$ for 
all $\tau \leq \alpha^{-2}$.  Let us differentiate $I_\alpha$:
\begin{equation*}
I_\alpha'  
=  - \frac{\omega' }{1-\omega } I_\alpha + (1 - \omega ) \int_{\mathcal{C}_\tau} \mu \left[ p \partial_\tau \eta_\alpha + 
\eta_\alpha p_\tau\right] dxd\theta.
\end{equation*}
Using~\eqref{e:p_equation} and (\ref{sep430}) allows us to 
rewrite the integral involving $p_\tau$:
\begin{equation}\label{sep604}
I_\alpha'(\tau)
=  \left(\aleph(\alpha) - 
\frac{\omega'(\tau)}{1-\omega(\tau)}\right) I_\alpha(\tau) - 
\omega(\tau) \int_{\mathcal{C}_\tau} \mu p \partial_\tau \eta_\alpha \, dxd\theta.
\end{equation}
The last term may be estimated as
\begin{equation}\label{eq:split_integral}
\begin{split}
\omega \left|\int_{\cC_\tau} \mu p \partial_\tau \eta_\alpha dxd\theta\right|
&\leq C\omega \int_{c^*\tau}^{c^*\tau + \alpha^{-1}} \int_\Theta p dx d\theta + 
C\omega \alpha \int_{c^*\tau + 
\alpha^{-1}}^\infty \int_\Theta p\eta_\alpha dx d\theta.
\end{split}
\end{equation}
The second term in~\eqref{eq:split_integral} is   
bounded by $C(\tau+1)^{-1} \alpha I_\alpha$.  
The first requires a bit more work.  First, split the integral as
\begin{equation}\label{eq:sep901}
\int_{c^*\tau}^{c^*\tau + \alpha^{-1}} \int_\Theta p dx d\theta
= \int_{c^*\tau}^{c^*\tau + \min\{\tau^{2/3},\alpha^{-1}\}} 
\int_\Theta p dx d\theta
+ \int_{c^*\tau + \min\{\tau^{2/3},\alpha^{-1}\}}^{c^*\tau + \alpha^{-1}} 
\int_\Theta p dx d\theta.
\end{equation}
The first term is estimated using~\eqref{sep110} to obtain
\[
\int_{c^*\tau}^{c^*\tau + \min\{\tau^{2/3},\alpha^{-1}\}} \int_\Theta p dx d\theta
\leq \int_{0}^{\min\{\tau^{2/3},\alpha^{-1}\}} 
\int_\Theta \frac{C_0 x}{(\tau+1)^{3/2}} dx d\theta
\leq \frac{C_0}{(\tau+1)^{1/6}}.
\]
Arguing as in~\cite[Lemma~5.4]{HNRR12}, we may bound $p$ by the solution to~\eqref{e:p_equation}
in the whole cylinder $\Rm\times \Theta$.  Thus, the heat kernel bounds of, e.g.~\cite{Norris}, imply that
\begin{equation}\label{eq:weak_heat_kernel}
	p(\tau, x+c^*\tau,\cdot)
		\leq \frac{C_0 e^{- \frac{x^2}{C(\tau+1)}}}{\sqrt{\tau+1}},
\end{equation}
where $C_0$ is a constant depending on $p_0$.  Hence,
the second integral in~\eqref{eq:sep901} yields
\[
\int_{c^*\tau + \min\{\tau^{2/3},\alpha^{-1}\}}^{c^*\tau + \alpha^{-1}} 
\int_\Theta p dx d\theta
\leq C\int_{\min\{\tau^{2/3},\alpha^{-1}\}}^{\alpha^{-1}} 
\int_\Theta \frac{e^{- \frac{x^2}{C(\tau+1)}}}{\sqrt{\tau+1}} dx d\theta
\leq {C_0} e^{-(\tau+1)^{1/3}/C}.
\]
From~\eqref{sep702}, we see that $|\omega| \leq C(\tau + T)^{-1}$, with $T$ to be chosen.  This, along with the previous two inequalities and~\eqref{eq:split_integral}, implies that
\[
	\omega\left|\int_{\cC_\tau} \mu p \partial_\tau \eta_\alpha dxd\theta\right|
		\leq \frac{C_0}{(\tau+T)(\tau+1)^{1/6}}
		\leq \frac{C_0}{T^{1/12} (\tau+1)^{13/12}}.
\]
We used above the first
assumption on $\omega$ in (\ref{sep702}).
Hence, we obtain
\begin{equation}\label{sep806}
\left| I_\alpha' - \left(\aleph(\alpha) - \frac{\omega'}{1-\omega} + 
O(\alpha/(\tau+1))\right)I_\alpha\right| 
\leq \frac{C_0}{T^{1/12} (\tau+1)^{13/12}}.
\end{equation}
Integrating (\ref{sep806}), using, once again, (\ref{sep702}),
yields the inequality
\begin{equation}\label{sep810}
	C(\tau+1)^{C\alpha} e^{\aleph(\alpha)\tau} 
\left(I_\alpha(0) - \frac{C_0}{T^{1/12}}\right)
	\leq I_\alpha(\tau)
	\leq C(\tau+1)^{C\alpha} e^{\aleph(\alpha)\tau} 
\left(I_\alpha(0) + \frac{C_0}{T^{1/12}}\right).
\end{equation}
Using that $\tau \leq \alpha^{-2}$ and that $\aleph(\alpha) \sim \alpha^2$, 
by~\eqref{sep432}, we have that $\tau^{C\alpha}e^{\aleph(\alpha)\tau} \leq C$.  
Using this and choosing $T$ at least as large as $(2C_0 / I_\alpha(0))^{12}$ 
in~\eqref{sep810} finishes the proof of the first estimate 
in~\eqref{eq:strong_I_alpha}.  
We note that, for all $\alpha$, we have
\[
I_\alpha(0) \geq \int_{\mathcal{C}_\tau} x p_0(x,\theta) dxd\theta,
\]
so that our 
condition on $T$ can be made uniform~in~$\alpha$.

Now we consider $V_\alpha$.  Fix $N$ to be determined later and
assume that $\tau^{2/3}<N\alpha^{-1}$, the other case being treated
via a very similar computation. We decompose the integral as
\begin{equation}\label{eq:V_alpha_decomp}
V_\alpha(\tau)
= \int_{c^*\tau}^{c^*\tau + \tau^{2/3}}\int_\Theta \eta_{2\alpha} p q dxd\theta
+ \int_{c^*\tau + \tau^{2/3}}^{c^*\tau + 
N\alpha^{-1}}\int_\Theta \eta_{2\alpha} p q dx d\theta
+ \int_{c^*\tau + N\alpha^{-1}}^{\infty} \int_\Theta \eta_{2\alpha} p q dx d\theta.
\end{equation}
For the first integral, using \Cref{lem:weightzeta} and the definition of $q$, we 
may apply \eqref{sep110} to bound $p$ and~$q$ as $C(x-c^*\tau)(\tau+1)^{-3/2}$ 
and $C(\tau+1)^{-3/2}$, respectively.  Using also~\eqref{eq:exponential} to bound $\eta_{2\alpha}$ by a linear function:
\begin{equation}\label{sep2820}
\hbox{$\eta_{2\alpha}(\tau,x,\theta) \leq C e^{2\alpha N} (x-c^*\tau)$ 
on $[c^*\tau, c^*\tau + N\alpha^{-1}]$,}
\end{equation}
we get
\begin{equation}\label{eq:V_alpha_1}
\int_{c^*\tau}^{c^*\tau + \tau^{2/3}}\int_\Theta \eta_{2\alpha} p q dxd\theta
\leq \frac{Ce^{2\alpha N}}{(\tau+1)^3}
\int_{c^*\tau}^{c^*\tau + \tau^{2/3}} (x-c^*\tau)^2 dx
\leq \frac{C e^{2\alpha N}}{(\tau+1)}.
\end{equation}
For the second integral in (\ref{eq:V_alpha_decomp}), 
we use the same bounds for $q$ and $\eta_{2\alpha}$ but we bound $p$ with the 
Gaussian bound~\eqref{eq:weak_heat_kernel}.  This yields
\begin{equation}\label{eq:V_alpha_2}
\int_{c^*\tau + \tau^{2/3}}^{c^*\tau + N\alpha^{-1}}
\int_\Theta \eta_{2\alpha} p q dx d\theta
\leq \frac{C e^{2\alpha N}}{(\tau+1)^2}
\int_{\tau^{2/3}}^{N\alpha^{-1}}
  x e^{-\frac{x^2}{C(\tau+1)}} dx  
\leq \frac{C e^{2\alpha N}}{(\tau+1)} e^{- (\tau+1)^{1/3}/C}.
\end{equation}
For the last integral in (\ref{eq:V_alpha_decomp}), 
we use the bound $q$ and $p$ as in the last step and bound  
$\eta_{2\alpha}$ by~$C e^{2\alpha x} / \alpha$.  This yields
\begin{equation*}\label{eq:V_alpha_3}
\begin{split}
\int_{c^*\tau + N\alpha^{-1}}^{\infty} \int_\Theta \eta_{2\alpha} p q dx d\theta
&\leq \frac{C}{(\tau+1)^2} \int_{N\alpha^{-1}}^{\infty}  
\frac{ e^{2\alpha x - \frac{x^2}{C(\tau+1)}}}{\alpha} dx  .
\end{split}
\end{equation*}
Since $\tau \leq \alpha^{-2}$, we may choose $N$ such that
\[
\farc{N}{C \alpha (\tau+1))} - 2\alpha \geq \farc{1}{\alpha(\tau+1)}
\]
for all $\alpha$ sufficiently small.  Hence, we have  
\[\begin{split}
\frac{C}{(\tau+1)^2} \int_{N\alpha^{-1}}^{\infty} 
  \frac{ e^{2\alpha x - \frac{x^2}{C(\tau+1)}}}{\alpha} dx  
&\leq \frac{C}{(\tau+1)^2} \int_{N\alpha^{-1}}^{\infty}   
\frac{ e^{2\alpha x - \frac{Nx}{C\alpha(\tau+1)}}}{\alpha} dx  \\
&\leq \frac{C}{(\tau+1)^2} \int_{N\alpha^{-1}}^{\infty}   
\frac{ e^{- 
\frac{x}{\alpha(\tau+1)}}}{\alpha} dx  
\leq \frac{C}{(\tau+1)}.
\end{split}\]
Combining this bound with~\eqref{eq:V_alpha_1} and~\eqref{eq:V_alpha_2}, we have that, for all $\tau \leq \alpha^{-2}$,
\[
V_\alpha(\tau) \leq C(\tau+1)^{-1},
\]
which, in particular, implies the upper bound on $V_\alpha$ in 
(\ref{eq:strong_I_alpha}).

\subsection{The proof of Proposition~\ref{p:Nash}}\label{sec:nash}

Here we prove the Nash-type inequality on cylinders that we use above.  We point out that when the $L^2$ norm is small relative to the $L^1$ norm, this yields the same inequality as in $\R^{k}$.  The main point here is that using this inequality we see that solutions to the heat equation on $\R^k\times \Theta$ decay at the same rate as solutions to the heat equation in $\R^k$.

Our approach is similar to the one used in~\cite{FanKisRyz}.  However, some computational challenges arise since we lack an explicit formula for the solutions of $k+1$ order polynomials.  We note that, by extending $\phi$ if necessary and scaling, we may assume without loss of generality that $\Theta = [0,1]^d$.

First, we represent $\phi$ in terms of its Fourier series in the $\theta$ variable, and its Fourier transform in the $x$ variable. This yields
\begin{equation*}
\phi(x,\theta) = \sum_{n \in \Z^d} \int_{\R^k} \hat\phi_n(\xi)  e^{i\xi\cdot x} \cos\left(\pi n \theta\right)\frac{d \xi}{(2\pi)^\frac{k+d}{2}}, 
\end{equation*}
where 
\begin{equation*}
\hat\phi_n(\xi) := \int_\Theta \int_{\R^k} \phi(x,\theta) e^{-i\xi \cdot x } \cos\left(\pi n \theta\right) \frac{dxd\theta}{(2\pi)^\frac{k+d}{2}}
\end{equation*}
Before we continue, we note two things.  First, we have that
\begin{equation}\label{eq:fourier_l1}
	|\hat\phi_n(\xi)| \leq \|\phi\|_{1},
\end{equation}
Second, the Plancherel formula tells that
\begin{equation}\label{eq:plancherel}
	\|\phi\|_2^2 = \sum_n \int_{\R^k} |\hat\phi_n(\xi)|^2 d\xi,
~~~\text{ and that }~~~  
	\|\nabla \phi\|_2^2 = \sum_n \int_{\R^k} \left(|\xi|^2 + n^2\right) |\hat\phi_n(\xi)|^2 \frac{d\xi}{(2\pi)^\frac{k+d}{2}}.
\end{equation}

Fix a constant $\rho$ to be determined later.  We now decompose $\|\phi\|_2$ into outer and inner parts as
\begin{equation}\label{eq:nash_split}
\begin{split}
	\|\phi\|_2^2
		&= \sum_{|n|\leq \rho} \int_{B_\rho(0)} |\hat \phi_n(\xi)|^2 d\xi
		+  \sum_{|n|> \rho} \int_{\R^d} |\hat \phi_n(\xi)|^2 d\xi
		+ \sum_{n} \int_{B_\rho(0)^c} |\hat \phi_n(\xi)|^2 d\xi.
\end{split}
\end{equation}
The first term in~\eqref{eq:nash_split} may be bounded 
as
\begin{equation}\label{eq:nash_first}
\begin{split}
	\sum_{|n|\leq \rho} \int_{B_\rho(0)} |\hat \phi_n(\xi)|^2 d\xi
		&\leq \sum_{|n|\leq \rho} \int_{B_\rho(0)} \|\phi_n\|_1^2 d\xi
		\leq C \rho^k(\rho + 1)^d\|\phi_n\|_1^2
\end{split}
\end{equation}
The second and third terms in~\eqref{eq:nash_split} may be estimated in the same way so we show only the second term.  It can be bounded as:
\begin{equation}\label{eq:nash_second}
	\begin{split}
		\sum_{|n|> \rho} \int_{\R^d} |\hat \phi_n(\xi)|^2 d\xi
			&\leq \sum_{|n|> \rho} \int_{\R^d} \frac{1}{\rho^2} \left(|\xi|^2 + n^2\right)|\hat \phi_n(\xi)|^2 d\xi 
			\leq \frac{1}{\rho^2} \|\nabla \phi\|_2^2.
	\end{split}
\end{equation}
Combining~\eqref{eq:nash_first} and~\eqref{eq:nash_second} with~\eqref{eq:nash_split}, we obtain
\[
	\frac{1}{C}\|\phi\|_2^2 \leq \rho^k(\rho + 1)^d\|\phi\|_1^2 + \frac{1}{\rho^2} \|\nabla\phi\|_2^2,
\]

In the interest of legibility, we define the following constants
\begin{equation*}
	I = 
		\Vert \phi \Vert_1^2, 
	\qquad J = \Vert \nabla \phi \Vert_2^2,
		\qquad\text{ and } \qquad
	K = \Vert \phi \Vert_2^2,
\end{equation*}
and we re-write the above inequality as
\[
	\frac{1}{C}K \leq \rho^k(\rho^d + 1) I + \frac{1}{\rho^2} J.
\]
Define $X$ to be the quantity 
\[
	X \stackrel{\text{def}}{=} J^{\frac{1}{k+2}} \left( \frac{C I^{\frac{2-d}{k+2}}}{K}\right)^{\frac{1}{k+d}},
\]
and choose
\begin{equation*}
	\rho = \left( \frac{J}{I} \right)^{\frac{1}{k+2}}
\end{equation*}
in order to optimize this inequality.  Hence, the above inequality becomes
\[
	1\leq
		X^{k+d} + \alpha X^k,
\]
where we define
\[
	\alpha = C\left(\frac{I}{K}\right)^\frac{d}{k+d}.
\]
It is straight-forward to verify that this polynomial has exactly one positive root which must be at least as large as $(2(1 + \alpha))^{-1/k}$.  Hence, it follows that
\[
	X \geq \left(\frac{1}{2(1 + \alpha)}\right)^{1/k}
		\geq C\frac{1}{1 + \alpha^{1/k}}.
\]
Returning to our earlier notation, we obtain
\[
	J
		\geq \frac{C K \left( \frac{I}{K}\right)^{\frac{d-2}{k+d}}}{1 + \left(\frac{I}{K} \right)^\frac{d(k+2)}{k(k+d)}},
\]
Re-arranging this inequality and substituting in for $I$, $J$, and $K$ concludes the proof.

\bibliographystyle{abbrv}
\bibliography{refs-log}
\end{document}